\documentclass[12pt]{article}

\usepackage{times}
\usepackage{amsmath}
\usepackage{verbatim}
\usepackage{amsfonts}
\usepackage{amssymb}
\usepackage{color}
\usepackage{bbm}
\usepackage{amsthm}
\usepackage{amscd}
\usepackage[all]{xy}
\usepackage{multirow}
\usepackage{tikz-cd}
\usepackage{mathtools}
\usepackage{cite}
\usepackage{mathdots}

\DeclareMathOperator{\Hom}{Hom}
\DeclareMathOperator{\End}{End}

\DeclareMathOperator{\mmod}{-mod}
\DeclareMathOperator{\proj}{-proj}
\DeclareMathOperator{\grmod}{-grmod}

\DeclareMathOperator{\grproj}{-grproj}
\DeclareMathOperator{\dgmod}{-dgmod}
\DeclareMathOperator{\dgMod}{-dgMod}
\DeclareMathOperator{\stab}{-stab}
\DeclareMathOperator{\grstab}{-grstab}
\DeclareMathOperator{\dgstab}{-dgstab}

\let\amsamp=&

\newtheorem{theorem}{Theorem}[section]
\newtheorem{lemma}[theorem]{Lemma}
\newtheorem{proposition}[theorem]{Proposition}
\newtheorem{corollary}[theorem]{Corollary}
\theoremstyle{definition}
\newtheorem{definition}[theorem]{Definition}
\theoremstyle{remark}
\newtheorem*{remark}{Remark}
\newtheorem{example}[theorem]{Example}

\begin{document}

\title{The Differential Graded Stable Category of a Self-Injective Algebra}
\author{Jeremy R. B. Brightbill}
\date{July 17, 2019}

\maketitle

\begin{abstract}
Let $A$ be a finite-dimensional, self-injective algebra, graded in non-positive degree.  We define $A\dgstab$, the differential graded stable category of $A$, to be the quotient of the bounded derived category of dg-modules by the thick subcategory of perfect dg-modules.  We express $A\dgstab$ as the triangulated hull of the orbit category $A\grstab/ \Omega(1)$.  This result allows computations in the dg-stable category to be performed by reducing to the graded stable category.  We provide a sufficient condition for the orbit category to be equivalent to $A\dgstab$ and show this condition is satisfied by Nakayama algebras and Brauer tree algebras.  We also provide a detailed description of the dg-stable category of the Brauer tree algebra corresponding to the star with $n$ edges.
\end{abstract}


\section{Introduction}

If $A$ is a self-injective $k$-algebra, then $A\stab$, the stable module category of $A$, admits the structure of a triangulated category.  This category has two equivalent descriptions.  The original description is as an additive quotient:  One begins with the category of $A$-modules and sets all morphisms factoring through projective modules to zero.  More categorically, we define $A\stab$ to be the quotient of additive categories $A\mmod/A\proj$.  The second description, due to Rickard \cite{rickard1989derived}, describes $A\stab$ as a quotient of triangulated categories.  Rickard obtains $A\stab$ as the quotient of the bounded derived category of $A$ by the thick subcategory of perfect complexes (i.e., complexes quasi-isomorphic to a bounded complex of projective modules).  Once this result is known, the triangulated structure on $A\stab$ is an immediate consequence of the theory of triangulated categories.  When translated back into the additive description, the homological shift functor $[-1]$ inherited from $D^b(A\mmod)$ becomes identified with the Heller loop functor $\Omega$, which maps each module to the kernel of a projective cover.  The triangulated description provides a well-behaved technical framework for transferring information between $A\stab$ and the derived category, while the additive description allows computations of morphisms to be performed in $A\mmod$ rather than $D^b(A\mmod)$.  If $A$ is made into a graded algebra, analogous constructions produce two equivalent descriptions of the graded stable category $A\grstab$.

If $A$ is a dg-algebra, we use the triangulated description to define the differential graded stable category $A\dgstab$.  More precisely, $A\dgstab$ is defined to be the quotient of the derived category $D^b_{dg}(A)$ of dg-modules by the thick subcategory of perfect dg-modules.  

The most immediately interesting feature of the dg-stable category is the presence of non-trivial interactions between the grading data and the triangulated structure.  In $D^b_{dg}(A)$, the grading shift functor coincides with the homological shift functor, and so in $A\dgstab$ the grading shift functor $(-1)$ can be identified with $\Omega$.  This phenomenon does not occur in the graded stable category, since the grading shift and homological shift functors in $D^b(A\grmod)$ are distinct.

However, working with dg-modules introduces new complications.  The presence of dg-modules which do not arise from complexes of graded modules is an obstacle to obtaining a simple additive definition of $A\dgstab$, without which computation of morphisms becomes much harder, as it must be done in the triangulated setting.  In this paper we consider the problem of finding a simple additive description of $A\dgstab$.

The dg-stable category has been studied by Keller \cite{Keller2005}, using the machinery of orbit categories; our approach is motivated by his work.  In Section \ref{The Dg-Stable Category}, we consider the case where $A$ is a non-positively graded, finite-dimensional, self-injective algebra, viewed as a dg-algebra with zero differential.  There is a natural functor $A\grstab \rightarrow A\dgstab$ which is faithful but not full.  This is due to the fact that $X \cong \Omega X(1)$ for all $X \in A\dgstab$; the corresponding isomorphism almost never holds in $A\grstab$.  To recover the missing morphisms, we turn to the orbit category $\mathcal{C}(A):= A\grstab/\Omega(1)$.  The objects of $\mathcal{C}(A)$ are those of $A\grstab$, and the morphisms $X\rightarrow Y$ are finite formal sums of morphisms $X \rightarrow \Omega^nY(n)$ in $A\grstab$.  Orbit categories need not be triangulated, but Keller proves they can always be included inside a "triangulated hull".  We shall construct a fully faithful functor $F_A: \mathcal{C}(A) \rightarrow A\dgstab$ whose image generates $A\dgstab$ as a triangulated category and show that $A\dgstab$ is the triangulated hull of $\mathcal{C}(A)$.

$F_A$ is an equivalence of categories precisely when it identifies $\mathcal{C}(A)$ with a triangulated subcategory of $A\dgstab$.  This is in general not the case, as there is no natural way to take the cone of a formal sum of morphisms with different codomains.  In Section \ref{Essential Surjectivity}, we provide a sufficient condition for $F_A$ to be an equivalence and show that this condition is satisfied by self-injective Nakayama algebras.  An example for which $F_A$ is not an equivalence is also provided.

In the second half of this paper, we investigate the dg-stable category of non-positively graded Brauer tree algebras.  A Brauer tree is the data of a tree, a cyclic ordering of the edges around each vertex, a marked vertex (called the exceptional vertex) and a positive integer multiplicity associated to the exceptional vertex.  The data of a Brauer tree determines, up to Morita equivalence, an algebra whose composition factors reflect the combinatorial data of the tree.  We refer to Schroll \cite{schroll2018brauer} for a detailed introduction to the theory of Brauer tree algebras and their appearance in group theory, geometry, and homological algebra, but we mention here one application which is of particular relevance.  Khovanov and Seidel \cite{khovanov2002quivers} link the category $D^b_{dg}(A)$, where $A$ is a graded Brauer tree algebra on the the line with $n$ vertices, to the triangulated subcategory of the Fukaya category generated by a chain of knotted Lagrangian spheres.  The braid group acts on $D^b_{dg}(A)$ by automorphisms, and the category $A\dgstab$ can be viewed as the quotient of $D^b_{dg}(A)$ by this action.

In Section \ref{Brauer Tree Algebras}, we show that $\mathcal{C}(A)$ is equivalent to $A\dgstab$ for any Brauer tree algebra, and in Section \ref{The Dg-Stable Category of the Star with n Vertices} we provide a detailed description of $A\dgstab$ when $A$ corresponds to the star with $n$ edges and multiplicity one.

This paper was motivated by the work of Chuang and Rouquier \cite{chuang2017perverse}, who, for a symmetric algebra $A$, define an action on a set of parametrized tilting complexes in $D^b(A)$.  Collections of simple $A$-modules act on the tilting complexes via application of perverse equivalences.  Placing a non-positive grading on $A$, this action can be defined in the dg-stable category, leading to interesting combinatorial interactions between the grading data of the algebra and the homological structure of the module category.  We shall describe the structure of this action when $A$ is a Brauer tree algebra in an upcoming work; the detailed computations in Section \ref{The Dg-Stable Category of the Star with n Vertices} serve as the foundation of this effort.


\section{Notation and Definitions}
\label{Notations and Definitions}

\subsection{Triangulated Categories}

A \textbf{triangulated category} is the data of an additive category $\mathcal{T}$, an automorphism $[1]$ of $\mathcal{T}$ (called the \textbf{suspension} or \textbf{shift}), and a family of \textbf{distinguished triangles} $X \rightarrow Y \rightarrow Z \rightarrow X[1]$ obeying certain axioms.  For more details on the definition and theory of triangulated categories, see Neeman \cite{neeman2001triangulated}.

In a triangulated category, any morphism $X \xrightarrow{f} Y$ can be completed to a triangle $X \xrightarrow{f} Y \rightarrow Z \rightarrow X[1]$.  We refer to $Z$ as a \textbf{cone} of $f$; it is unique up to non-canonical isomorphism.  Abusing notation, for any morphism $f:X \rightarrow Y$ in a triangulated category $\mathcal{T}$, we shall write $C(f)$ to refer to any choice of object completing the triangle $X \xrightarrow{f} Y \rightarrow C(f) \rightarrow X[1]$.  This will cause no confusion.

An additive functor $F: \mathcal{T}_1 \rightarrow \mathcal{T}_2$ is said to be \textbf{exact} or \textbf{triangulated} if it commutes with the shift functor and preserves distinguished triangles.  A full additive subcategory $\mathcal{I}$ of $\mathcal{T}$ is called \textbf{triangulated} if it is closed under isomorphisms, shifts, and cones.  A triangulated subcategory $\mathcal{I}$ of $\mathcal{T}$ is \textbf{thick} if $\mathcal{I}$ is closed under direct summands.  Given a thick subcategory, one can form the quotient category $\mathcal{T}/\mathcal{I}$ by localizing at the class of morphisms whose cone lies in $\mathcal{I}$.  There is a natural functor $\mathcal{T} \rightarrow \mathcal{T}/\mathcal{I}$ which is essentially surjective and whose kernel is $\mathcal{I}$.

\subsection{Complexes}
\label{complexes}
If $\mathcal{A}$ is any additive category, we write $Comp(\mathcal{A})$ for the category of (cochain) complexes over $\mathcal{A}$.  We shall write our complexes as $(C^\bullet, d^\bullet_C)$, where $d^n_C: C^n \rightarrow C^{n+1}$ for all $n\in \mathbb{Z}$.  We write $Ho(\mathcal{A})$ for the category of complexes and morphisms taken modulo homotopy.  If $\mathcal{A}$ is an abelian category, we let $D(\mathcal{A})$ denote the derived category of $\mathcal{A}$.  On any of these subcategories, we shall use the superscript $b$ (resp., $+, -$) to denote the full, replete subcategory generated by the bounded (resp., bounded below, bounded above) complexes.

If $\mathcal{A} = A\mmod$ for some algebra $A$, we shall write $Ho^{perf}(A\mmod)$ (resp., $D^{perf}(A\mmod)$) for the thick subcategory of $Ho(A\mmod)$ (resp., $D(A\mmod)$) generated by $A$.  The objects are those complexes which are homotopy equivalent  (resp., quasi-isomorphic) to bounded complexes of projective modules; we refer to them as the \textbf{strictly perfect}  (resp., \textbf{perfect}) complexes.

On any of the above categories, we let $[n]$ denote the $n$-th shift functor, defined by $(C^\bullet [n])^i = C^{i+n}$ and $d^\bullet_{C[n]} = (-1)^{n}d^\bullet_{C}$.  We write $H^n(C^\bullet)$ for the $n$-th cohomology group of $C^\bullet$, i.e. $H^n(C^\bullet) := ker(d^n_C)/im(d^{n-1}_C)$.

Given a morphism of complexes $f: X^\bullet \rightarrow Y^\bullet$, we define the \textbf{cone} of $f$ to be the complex $C(f)^\bullet = X^\bullet[1] \oplus Y^\bullet$ with differential $\begin{pmatrix} d^{\bullet}_{X[1]} & 0 \\ f[1] & d^\bullet_Y \end{pmatrix}$.  We obtain an exact triangle $X \xrightarrow{f} Y \rightarrow C(f) \rightarrow X[1]$ in $Ho(\mathcal{A})$.

We write $\tau_{\le n}, \tau_{\ge n}, \tau_{< n}, \tau_{> n}$ for the truncation functors on $D(\mathcal{A})$ defined by the canonical t-structure.  More explicitly, if $X^\bullet$ is a complex, the $k$th term of $\tau_{\le n}X^\bullet$ is $X^k$ if $k < n$, $0$ if $k > n$, and $ker(d^n_X)$ if $k =n$.  $\tau_{> n}X^\bullet$ is defined analogously; here the $n$th term equal to $im(d^n_X)$.  We also denote by $X^{\le n}$ the complex whose $k$th term is $X^k$ for $k \le n$ and $0$ for $k > n$.  We denote $X^{\ge n}, X^{< n}, X^{>n}$ similarly, and refer to these complexes as the \textbf{sharp truncations} of $X^\bullet$.


\subsection{Modules and the Stable Category}
\label{module definitions}

If $A$ is an algebra over a field $k$, we let $A\mmod$ denote the category of finitely generated right $A$-modules, and let $A\proj$ denote the full subcategory of finitely generated projective modules.   $A$-Mod and $A$-Proj will denote the categories of all modules and projective modules, respectively.

Given an $A$-module $X$, we define the \textbf{socle} of $X$, $soc(X)$ to be the sum of all simple submodules $X$.  We define the \textbf{radical} of $X$, $rad(X)$, to be the intersection of all maximal submodules of $X$, and we define the \textbf{head} of $X$ to be the quotient $hd(X) = X/rad(X)$.  We note that $rad(A)$, where $A$ is viewed as a right module over itself, is equal to the Jacobson radical of $A$.  If $X$ is finitely generated, then $rad(X) = Xrad(A)$.

An algebra $A$ is said to be \textbf{self-injective} if $A$ is injective as a right $A$-module.  In a self-injective algebra, the classes of finitely-generated projective and injective modules coincide.  $A$ is said to be \textbf{symmetric} if there is a linear map $\lambda: A \rightarrow k$ such that $ker(\lambda)$ contains no left or right ideals of $A$, and $\lambda(ab) = \lambda(ba)$ for all $a, b \in A$.  All symmetric algebras  are self-injective.

We let $A\stab$ denote the \textbf{stable module category} of $A$.  The objects of $A\stab$ are the objects of $A\mmod$, and $\Hom_{A\stab}(X,Y)$ is defined to be the quotient of $\Hom_{A\mmod}(X,Y)$ by the subspace of morphisms factoring through projective modules.  There is a full, essentially surjective functor $A\mmod \twoheadrightarrow A\stab$ which is the identity on objects.  If $A$ is self-injective, then $A\stab$ admits the structure of a triangulated category, and it has been shown by Rickard \cite{rickard1989derived} that $A\stab$ is equivalent as a triangulated category to $D^b(A\mmod)/D^{perf}(A\mmod)$.

For a more detailed introduction to the theory of finite-dimensional algebras, see for instance Benson \cite{benson1991representations}.

We say that an additive category $\mathcal{C}$ is a \textbf{Krull-Schmidt category} if every object in $\mathcal{C}$ is isomorphic to a finite direct sum of objects with local endomorphism rings.

If $A$ is finite-dimensional, both $A\mmod$ and $A\stab$ are Krull-Schmidt.  Any indecomposable object of $A\stab$ is the image of an indecomposable object in $A\mmod$.


\subsection{Graded Modules}
Let $A$ be a graded algebra over a field $k$.  We denote by $A\grmod$ (resp., $A\grproj$) the category of finitely generated graded right modules (resp., finitely generated graded projective right modules).  We shall use upper case letters when the modules are not required to be finitely generated, in analogy with Section \ref{module definitions}.

The \textbf{graded stable module category} $A\grstab$ can be defined analogously to $A\stab$.  When $A$ is self-injective, we once again have that $A\grstab$ is triangulated and equivalent to $D^b(A\grmod)/D^{perf}(A\grmod)$.

If $X$ is a graded $A$-module, we write $X^i$ to denote the homogenous component of $X$ in degree $i$.  (If $X^\bullet$ is a complex of graded modules, we shall denote the degree $i$ component of the $n$th term of the complex by $(X^n)^i$.)  On any category of graded objects, we define the grading shift functor $(n)$ by $X(n)^i = X^{i + n}$.  If $x \in X$ is a homogeneous element, we let $|x|$ denote the degree of $x$.

For a graded module $X$, we define the \textbf{support} of $X$ to be the set $supp(X) = \{n \in \mathbb{Z} | X^n \neq 0\}$.  We also define $max(X) = sup(supp(X))$ and $min(X) = inf(supp(X))$.  Note that if $A$ is finite-dimensional and $X$ is a finitely generated nonzero $A$-module, then $X$ is a finite-dimensional $k$-vector space, therefore $supp(X)$ is a finite, nonempty set and $max(X)$ and $min(X)$ are finite.

Given graded modules $X$ and $Y$, define $\Hom^{\bullet}_{A\grmod}(X, Y)$ to be the graded vector space whose degree $n$ component is the space $\Hom_{A\grmod}(X, Y(n))$ of degree $n$ morphisms.  If $X$ is a graded left $B$-module for some graded algebra $B$, then $\Hom^{\bullet}_{A\grmod}(X, Y)$ is a graded right $B$-module.


\subsection{Differential Graded Modules}
\label{Differential Graded Modules}

A \textbf{differential graded algebra} is a pair $(A, d_A)$, where $A$ is a graded $k$-algebra and $d_A$ is a degree $1$ $k$-linear differential which satisfies, for all homogenous $a, b \in A$, the equation $d_A(ab) = d_A(a)b + (-1)^{|a|}ad_A(b)$ .

If $(A, d_A)$ is a differential graded $k$-algebra, a \textbf{differential graded right $A$-module} (or dg-module, for short) is a pair $(X, d_X)$ consisting of a graded right $A$-module $X$ and a degree $1$ $k$-linear differential $d_X: X\rightarrow X$ satisfying $d_X(xa) = d_X(x)a + (-1)^{|x|}xd_A(a)$ for all homogeneous elements $x\in X, a \in A$.  A morphism of differential graded modules is defined to be a homomorphism of graded $A$ modules which commutes with the differential.  We denote by $A\dgmod$ the category of finitely-generated right dg-modules.  As above, we shall write $A\dgMod$ for the category of arbitrary dg-modules.

As with complexes, we write $H^i(X)$ for the $i$th cohomology group of $X$, i.e. the degree $i$ component of $ker(d_X)/im(d_X)$.  

Any graded algebra $A$ can be viewed as a differential graded algebra with zero differential.  In this case, for any dg-module $(X, d_X)$, $d_X$ is a degree $1$ morphism of graded modules, and so the kernel and image of $d_X$ are dg-submodules of $X$ with zero differential.  In this paper, we shall work exclusively with dg-algebras with zero differential.

For dg-modules, we define the grading shift functor $(n): A\dgMod \rightarrow A\dgMod$ by $(X, d_X)(n) = (X(n), d_{X(n)})$, where $d_{X(n)} = (-1)^n d_X(n)$. 

There is a faithful functor $\widehat{\phantom{x}}: Comp(A\grmod) \rightarrow A\dgMod$ sending the complex $(X^\bullet, d^{\bullet}_X)$ to the dg-module $(\widehat{X}, d_{\widehat{X}})$ whose underlying graded module is $\widehat{X} = \bigoplus_{n\in \mathbb{Z}} X^n(-n)$ and whose differential $d_{\widehat{X}}$ restricts to $d^n_{X}(-n)$ on $X^n(-n)$.  If $A$ is finite-dimensional, $\widehat{\phantom{x}}$ restricts to a functor $Comp^b(A\grmod) \rightarrow A\dgmod$.  Note also that $\widehat{X^\bullet [k]} = \widehat{X}(k)$.

Identifying graded modules with complexes concentrated in degree zero yields a fully faithful functor $A\grmod \hookrightarrow Comp^b(A\grmod)$.  The restriction of $\widehat{\phantom{x}}$ to $A\grmod$ is fully faithful.  Note that $\widehat{X(k)} = \widehat{X}(k)$.

If $f, g: X \rightarrow Y$ are morphisms of dg-modules, we say $f$ and $g$ are \textbf{homotopic} if there is a degree $-1$ graded morphism $h: X \rightarrow Y$ such that $f-g = h\circ d_X + d_Y \circ h$.  We write $Ho_{dg}(A)$ for the category of left dg-modules over $A$ and homotopy classes of morphisms.  By formally inverting the quasi-isomorphisms of $Ho_{dg}(A)$, we obtain $D_{dg}(A)$, the derived category of dg-modules.  We again use the superscript $b$ (resp., $+, -$) to denote the full subcategory whose objects are isomorphic to dg-modules with bounded (resp. bounded below, bounded above) support.  We write $Ho^{perf}_{dg}(A)$ (resp., $D^{perf}_{dg}(A)$) for the thick subcategories of $Ho^b_{dg}(A)$ (resp., $D^b_{dg}(A)$) generated by the dg-module $A$.  We refer to the objects of $Ho^{perf}_{dg}(A)$ (resp., $D^{perf}_{dg}(A)$) as the \textbf{strictly perfect} (resp., \textbf{perfect}) dg-modules.

If $P$ is strictly perfect, then for any dg-module $X$, we have an isomorphism $\Hom_{Ho_{dg}(A)}(P, X) \cong \Hom_{D_{dg}(A)}(P, X)$.  In addition, if $A$ is a finite-dimensional, self-injective graded algebra with zero differential, then $\Hom_{Ho_{dg}(A)}(X, P) \cong \Hom_{D_{dg}(A)}(X, P)$.  Any perfect dg-module is quasi-isomorphic to a strictly perfect dg-module.

If $P^\bullet \in Comp^-(A\grproj)$, then $\Hom_{Ho_{dg}(A)}(\widehat{P}, X) \cong \Hom_{D_{dg}(A)}(\widehat{P}, X)$ for all $X\in A\dgmod$.  If $A$ is finite-dimensional, self-injective, and has zero differential, then for $I^\bullet \in Comp^+(A\grproj)$, we also have $\Hom_{Ho_{dg}(A)}(X, \widehat{I}) \cong \Hom_{D_{dg}(A)}(X, \widehat{I})$ for all $X\in A\dgmod$.

We define the \textbf{dg-stable module category} of $A$ to be the quotient $A\dgstab:= D^b_{dg}(A)/D^{perf}_{dg}(A)$.  

If $A$ is finite-dimensional, we have essentially surjective functors $A\dgmod \twoheadrightarrow Ho^b_{dg}(A) \twoheadrightarrow D^b_{dg}(A) \twoheadrightarrow A\dgstab$, each of which is the identity on objects.  By composing with the inclusion $A\grmod \hookrightarrow A\dgmod$, we obtain an additive functor $A\grmod \rightarrow A \dgstab$ whose kernel contains $A\grproj$.  Hence this functor factors through $A\grmod \twoheadrightarrow A\grstab$.

Given a morphism of dg-modules $f: X \rightarrow Y$, we define the \textbf{cone} of $f$ to be the complex $C(f) = X(1) \oplus Y$ with differential $\begin{pmatrix} d_{X(1)} & 0 \\ f(1) & d_Y \end{pmatrix}$.  We obtain an exact triangle $X \xrightarrow{f} Y \rightarrow C(f) \rightarrow X(1)$ in $Ho_{dg}(A)$.


\subsection{Functors and Resolutions} \label{resolutions}

Let $A$ be a finite-dimensional, self-injective graded algebra.

Let $m: A^{op} \otimes_k A \twoheadrightarrow A$ denote the multiplication map, viewed as a morphism of graded $(A^{op}\otimes_k A)$-modules, and let $I = ker(m)$.  We define the functor $\Omega := - \otimes_A I: A\grmod \rightarrow A\grmod$.  Note that $\Omega$ has a right adjoint $\Omega' := \Hom^{\bullet}_{A\grmod}(I, -)$.

Since $I$ is projective both as a right and left $A$-module, we have that $\Omega$ is exact and $\Omega(A\grproj) \subset A\grproj$.  Thus $\Omega$ lifts to $D^b(A\grmod)$ and descends to $A\grstab$; we also have that $\Omega(D^{perf}(A\grmod)) \subset D^{perf}(A\grmod)$.  The complex $P^\bullet = 0 \rightarrow I \hookrightarrow A^{op} \otimes_k A \twoheadrightarrow A \rightarrow 0$ is an exact complex of projective right $A$-modules, hence is homotopy equivalent to zero.  Then for any $X \in A\grmod$, we have that $X \otimes_A P^\bullet $ is homotopy equivalent to zero, hence exact.  But $X \otimes_A P^\bullet \cong 0 \rightarrow \Omega X \hookrightarrow  X \otimes_k A \twoheadrightarrow X \rightarrow 0$, hence $\Omega X$ is the kernel of a surjection from a projective right $A$-module onto $X$.  Thus $\Omega$ is an autoequivalence of $A\grstab$ and is isomorphic to the desuspension functor for the triangulated structure.  In $A\grstab$, $\Omega X$ is isomorphic to the kernel of a projective cover of $X$.

Similarly, for any complex $X^\bullet \in Comp^b(A\grmod)$, we have a short exact sequence of complexes $0 \rightarrow \Omega(X^\bullet) \hookrightarrow X^\bullet \otimes_k A \twoheadrightarrow X^\bullet \rightarrow 0$.  From the resulting triangle in $D^b(A\grmod)$, we obtain a natural transformation $[-1] \rightarrow \Omega$.  Since $X^\bullet \otimes_k A \in D^{perf}(A\grmod)$, this natural transformation descends to a natural isomorphism $[-1] \rightarrow \Omega$ in $A\grstab \cong$ $D^b(A\grmod)/D^{perf}(A\grmod)$.

By a similar argument, $\Omega$ defines a functor $A\dgmod \rightarrow A\dgmod$ which is exact and preserves direct summands of $A$.  Thus $\Omega$ lifts to $D^b_{dg}(A)$ and preserves $D^{perf}_{dg}(A)$, and so $\Omega$ descends to $A\dgstab$.  We also have a natural transformation $(-1) \rightarrow \Omega$ of endofunctors of $D^b_{dg}(A)$ which descends to an isomorphism in $A\dgstab$.

Similarly, $\Omega'$ is exact and preserves projective modules, and so descends to $A\grstab$ and $A\dgstab$ and lifts to $D^b(A\grmod)$ and $D^b_{dg}(A)$.  Since $\Omega'$ is right adjoint to $\Omega$, we have that $\Omega'$ is quasi-inverse to $\Omega$ in $A\grstab$ and $A\dgstab$.

For any $X\in A\grmod$, we can construct a projective resolution $(P^\bullet_X, d^{\bullet}_{P_X})$ of $X$, such that $coker(d^{-n-1}_{P_X}) = \Omega^{n}(X)$ for any $n \ge 0$.  More specifically, for $n\ge 0$, we let $P^{-n}_X = \Omega^nX \otimes_k A$,  and for $n \ge 1$ we let $d^{-n}_{P_X}$ be the composition $P^{-n}_X = \Omega^nX \otimes_k A \twoheadrightarrow \Omega^n X \hookrightarrow P^{-n+1}_X$.  Likewise, we can construct an injective resolution $(I^\bullet_X, d^{\bullet}_{I_X})$ such that $ker(d^n_{I_X}) = (\Omega')^n(X)$ for all $n\ge 0$.  We let $I^n_X = \Hom^\bullet_{A\grmod}(A^{op}\otimes_k A, (\Omega')^n(X))$ and define the differential analogously.  Joining $P^\bullet_X$ and $I^\bullet_X$ via the map $P^0_X \twoheadrightarrow X \hookrightarrow I^0_X$, we can define an acyclic biresolution $(B^\bullet_X, d^\bullet_{B_X})$ with $B^n_X = I^n_X$ for $n \ge 0$ and $B^n_X = P^{n+1}_X$ for $n < 0$.  We refer to these resolutions as the \textbf{standard resolutions} of $X$.


\subsection{Dg-categories}

We give a brief introduction to the terminology and machinery of dg-categories.  For more details, the reader may consult Keller \cite{keller2006differential}.

A \textbf{differential graded category} or \textbf{dg-category} over a field $k$ is a category enriched over $k\dgMod$.  A \textbf{dg-functor} between two dg-categories is a functor of categories enriched over $k\dgMod$.  Thus morphism spaces in a dg-category have a natural structure of complexes of $k$-vector spaces, and dg-functors induce morphisms of complexes on these Hom spaces.  Let $dgcat_k$ denote the category of all small dg-categories over the field $k$.

Given a dg-category $\mathcal{A}$, we define the \textbf{homotopy category} of $\mathcal{A}$ to be the category $H^0(\mathcal{A})$ whose objects are the objects of $\mathcal{A}$ and whose morphsms are given by $\Hom_{H^0(\mathcal{A})}(X,Y) := H^0(\Hom_{\mathcal{A}}(X,Y))$.  Similarly, we define $Z^0(\mathcal{A})$ to be the category with the same objects whose morphisms are the closed degree 0 morphisms of $\mathcal{A}$.

A dg-functor $F: \mathcal{A} \rightarrow \mathcal{B}$ is called a \textbf{quasi-equivalence} if $F$ induces quasi-isomorphisms on all morphism spaces and $H^0(F): H^0(\mathcal{A}) \rightarrow H^0(\mathcal{B})$ is an equivalence.  By inverting the quasi-equivalences in $dgcat_k$, we obtain the \textbf{homotopy category of dg-categories}, denoted $Ho(dgcat_k)$.

Define the \textbf{dg-category of dg $k$-modules} $k\dgMod_{dg}$ as follows:  The objects of $k\dgMod_{dg}$ are the differential graded $k$-modules.  The degree $n$ piece of $\Hom_{k\dgMod_{dg}}(X, Y)$ is $\bigoplus_{i} \Hom_k(X^i, Y^{i+n})$, and the differential is given by $d_n(f_i) = d_{Y^{n+i}}f_i + (-1)^n f_i d_{X^{i-1}}$.

Given a dg-category $\mathcal{A}$, we define a \textbf{right dg $\mathcal{A}$-module} to be a dg-functor $M: \mathcal{A}^{op} \rightarrow k\dgMod_{dg}$.  A morphism of dg $\mathcal{A}$-modules is a natural transformation of dg-functors.  We let $\mathcal{A}\dgMod$ denote the category of dg $\mathcal{A}$-modules.  Moreover, we can define the \textbf{dg-category of dg $\mathcal{A}$-modules}, $\mathcal{A}\dgMod_{dg}$ in analogy to $k\dgmod_{dg}$.

Each object $X$ of $\mathcal{A}$ defines a dg $\mathcal{A}$-module $X^{\wedge}: Y \mapsto \Hom_\mathcal{A}(Y, X)$, and $\wedge:  X \mapsto X^{\wedge}$ defines a fully faithful dg-functor from $\mathcal{A}$ to $\mathcal{A}\dgMod_{dg}$.  We say a dg-module $M$ is \textbf{representable} if it is isomorphic to a dg-module in the image of $\wedge$.  We can define the shift $M[n]$ of a dg-module $M$ by $M[n](X) = M(X)[n]$.  Similarly, we define the cone of a morphism of dg-modules $f: M \rightarrow N$ to be the dg-module given by $C(f)(X) = C(f_X)$.  The homology of a dg-module is also defined object-wise.

Note that when $\mathcal{A}$ has one object, $*$, we can identify $\mathcal{A}$ with the dg-algebra $A = \End_\mathcal{A}(*)$; in this case a dg $\mathcal{A}$-module is the same data as a dg $A$-module.

Localizing $\mathcal{A}\dgMod$ at the quasi-isomorphisms, we obtain the \textbf{derived category} $D_{dg}(\mathcal{A})$ of dg $\mathcal{A}$-modules.  Each dg-functor $F:  \mathcal{A} \rightarrow \mathcal{B}$ defines an $\mathcal{A}^{op} \otimes \mathcal{B}$-module given by $X_F: (A, B) \mapsto \Hom_\mathcal{B}(B, FA)$.  Denote by $rep(\mathcal{A}, \mathcal{B})$ the set of objects $X \in D_{dg}(\mathcal{A}^{op} \otimes \mathcal{B})$ such that for all $A \in \mathcal{A}$, $X(A, -)$ is isomorphic in $D_{dg}(\mathcal{B})$ to a representable $\mathcal{B}$-module.  The correspondence $F \mapsto X_F$ defines a bijection between the morphisms $\mathcal{A} \rightarrow \mathcal{B}$ in $Ho(dgcat_k)$ and the isomorphism classes of objects in $rep(\mathcal{A}, \mathcal{B})$.  This correspondence respects composition, where composition of bimodules is given by the derived tensor product.

A dg-category $\mathcal{A}$ is \textbf{pretriangulated} if for every $X \in \mathcal{A}$ and $n \in \mathbb{Z}$, the dg-module $X^\wedge[n]$ is representable, and for any closed morphism $f$ of degree zero, $C(f^\wedge)$ is representable.  If $\mathcal{A}$ is pretriangulated, then $H^0(\mathcal{A})$ is a triangulated category, with shift and cone induced from $\mathcal{A}$.  Every dg-category $\mathcal{A}$ embeds into a \textbf{pretriangulated hull} $PreTr(\mathcal{A})$.  The pretriangulated hull has an explicit construction, due to Bondal and Kapranov \cite{bondal1991enhanced}, in terms of one-sided twisted complexes.  (Note that \cite{bondal1991enhanced} uses the notation $PreTr^+(\mathcal{A})$ for this construction.)  Every morphism in $Ho(dgcat_k)$ from a dg-category $\mathcal{A}$ to a pretriangulated dg-category $\mathcal{B}$ factors uniquely through $PreTr(\mathcal{A})$, and if $\mathcal{A}$ is already pretriangulated then it is quasi-equivalent to $PreTr(\mathcal{A})$.  Define the \textbf{triangulated hull} of $\mathcal{A}$ to be the triangulated category $Tr(\mathcal{A}):= H^0(PreTr(\mathcal{A}))$.  $H^0(\mathcal{A}) \hookrightarrow Tr(\mathcal{A})$ generates $Tr(\mathcal{A})$ as a triangulated category; if $\mathcal{A}$ is pretriangulated, then $H^0(\mathcal{A})$ is equivalent to $Tr(\mathcal{A})$.

Given a dg-category $\mathcal{A}$ and a full dg-subcategory $\mathcal{B}$, there is a dg-category $\mathcal{A}/\mathcal{B}$, called the \textbf{dg-quotient of $\mathcal{A}$ by $\mathcal{B}$}.  The dg-quotient is obtained by "setting all objects of $\mathcal{B}$ to zero" by formally adjoining a morphism $\mathcal{E}_B$ for each object $B\in \mathcal{B}$ satisfying $d(\mathcal{E}_B) = id_B$.  (See Drinfeld \cite{drinfeld2004dg} for details.)  The dg-quotient generalizes the Verdier quotient in the sense that $Tr(\mathcal{A})/Tr(\mathcal{B}) \cong Tr(\mathcal{A}/\mathcal{B})$.  It is characterized by the universal property that any morphism $\mathcal{A} \rightarrow \mathcal{C}$ in $Ho(dgcat_k)$ sending each object of $\mathcal{B}$ to a contractible object in $\mathcal{C}$ factors uniquely through $\mathcal{A}/\mathcal{B}$ (Tabuada, \cite{tabuada2010drinfeld}).

Given a dg-category $\mathcal{A}$ and $F \in rep(\mathcal{A}, \mathcal{A})$, one can define the dg-orbit category $P: \mathcal{A} \rightarrow \mathcal{A}/F$.  This category is characterized by the universal property that for any dg-category $\mathcal{B}$ and any $G \in rep(\mathcal{A}, \mathcal{B})$ such that $G \circ F \cong G$, $G$ factors through $P$.  For a more precise description, see Keller \cite{Keller2005}, Section 9.3.


\subsection{Autoequivalences and Automorphisms} \label{autoequivalence}
Given any category $\mathcal{C}$ and an autoequivalence $F: \mathcal{C} \rightarrow \mathcal{C}$, there is a category $\widetilde{\mathcal{C}}$, an automorphism $\widetilde{F}: \widetilde{\mathcal{C}} \rightarrow \widetilde{\mathcal{C}}$, and an equivalence of categories $\pi: \widetilde{\mathcal{C}} \rightarrow \mathcal{C}$ such that $\pi \circ \widetilde{F} = F \circ \pi$.  Identifying $\mathcal{C}$ with $\widetilde{\mathcal{C}}$, we can assume without loss of generality that the autoequivalence $F$ of $\mathcal{C}$ is an automorphism.


\subsection{Left Dg-Modules}

In this paper, we work exclusively with right dg-modules.  All the results presented are valid for left dg-modules, but minor adjustments must be made to account for numerous unpleasant sign conventions.  We describe the necessary adjustments here.

If $(A, d_A)$ is a differential graded $k$-algebra, we define a \textbf{left differential graded $A$-module} (or dg-module, for short) to be a pair $(X, d_X)$ consisting of a graded left $A$-module $X$ and a degree $1$ $k$-linear differential $d_X: X\rightarrow X$ satisfying $d_X(ax) = d_A(a)x + (-1)^{|a|}ad_X(x)$ for all homogeneous $a\in A, x \in X$.  We let $A\dgmod_l$ denote the category of left dg-modules over $A$.

If $A$ is a graded algebra, we define the algebra $(\overline{A}, \circ)$ to be the set $A$ with multiplication given by $a\circ b = (-1)^{|a||b|}(ab)$.  Similarly, if $M$ is a graded right $A$-module, we denote by $(\overline{M}, \circ)$ the graded right $\overline{A}$-module with $M$ as the underlying set and the operation given by $m \circ a = (-1)^{|m||a|}ma$.  We define $\overline{M}$ similarly for left graded modules.  The functor sending $M$ to $\overline{M}$ and acting as the identity on morphisms defines an isomorphism between $A\grmod$ and $\overline{A}\grmod$.

Let $A^{op}$ denote the opposite algebra.  We call $\overline{A}^{op}$ the graded opposite algebra.  If $(A, d_A)$ is a dg-algebra, then $(\overline{A}^{op}, d_A)$ is also a dg-algebra.  If $(M, d_M) \in A\dgmod$, then $(\overline{M}, d_M) \in \overline{A}^{op}\dgmod_l$, and this defines an isomorphism of categories.

As before, there is a faithful functor $\widehat{\phantom{x}}: Comp(A\grmod_l) \rightarrow A\dgMod_l$ sending the complex $(M^\bullet, d^{\bullet}_M)$ to the dg-module $(\widehat{M}, d_{\widehat{M}})$.  The underlying graded module $(\widehat{M}, *)$ is given by $\widehat{M} = \bigoplus_{n\in \mathbb{Z}} M^n(-n)$; the operation $*$ is defined by $a*m = (-1)^{|a|n}am$, where $m \in M^n(-n)$.  The differential $d_{\widehat{M}}$ restricts to $d^n_{M}(-n)$ on $M^n(-n)$.  This definition of $\widehat{\phantom{a}}$ is equivalent to converting to complexes of right $\overline{A}^{op}$-modules, applying the original definition of $\widehat{\phantom{a}}$, and then converting back to left dg-modules over $A$.

If $M$ is a left dg-module, define the dg-grading shift functor $\langle n \rangle: (M, d_M) \mapsto (M\langle n \rangle, d_{M\langle n \rangle})$.  The underlying set of the left graded module $(M\langle n \rangle, \cdot_n)$ is $M(n)$, and the operation $\cdot_n$ is given by $a \cdot_n m = (-1)^{|a|n}am$.  The differential is given by $d_{M\langle n \rangle} = (-1)^n d_M$.  Triangles in the homotopy or derived categories take the form $X \rightarrow Y \rightarrow Z \rightarrow X\langle 1 \rangle$.  For $M^{\bullet} \in Comp(A\grmod_l)$ we have that $\widehat{M^{\bullet}[n]} = \widehat{M}\langle n \rangle$ and $\widehat{M^{\bullet}(n)} = \widehat{M}(n)$.

If $X$ and $Y$ are graded modules, we say that a function $f: X\rightarrow Y$ is a \textbf{graded skew-morphism of degree $n$} if it is a degree $n$ $k$-linear map such that $f(ax) = (-1)^{n|a|}af(x)$ for all $x\in X$ and all homogeneous $a \in A$.  We say two morphisms of left dg-modules $f, g: X\rightarrow Y$ are \textbf{homotopic} if there is a graded skew-morphism $h: X \rightarrow Y$ of degree $-1$ such that $f-g = h\circ d_X + d_Y\circ h$.  We also note that if $A$ has zero differential, then $d_A$ is a graded skew-morphism of degree $1$.


\section{The Dg-Stable Category}
\label{The Dg-Stable Category}

\subsection{Construction}

Let $A$ be a finite-dimensional, non-positively graded, self-injective $k$-algebra, viewed as a dg-algebra with zero differential.  In this section, we shall provide a description of the dg-stable category of $A$ in terms of the graded stable category.  In Definition \ref{nice category}, we define the orbit category $\mathcal{C}(A) = A\grstab/\Omega(1)$.  In Definition \ref{inclusion functor} we define a functor $F_A: \mathcal{C}(A) \rightarrow A\dgstab$ and in Theorem \ref{dgstable} we show that $F_A$ is fully faithful with essential image generating $A\dgstab$ as a triangulated category.

We begin with some simple facts about graded $A$-modules.

\begin{proposition}
Let $X, Y \in A\grmod$.  If $supp(X) \cap supp(Y) = \emptyset$, then both $\Hom_{A\grmod}(X, Y)$ and $\Hom_{A\grstab}(X, Y)$ are zero.  
\end{proposition}
\begin{proof}
The first part of the statement follows immediately from the definition of morphisms of graded modules.  Since $\Hom_{A\grstab}(X, Y)$ is defined as a quotient of $\Hom_{A\grmod}(X, Y)$, the second part of the statement follows from the first.
\end{proof}

\begin{proposition}
\label{head lemma}
Let $X \in A\grmod$.  Then
\begin{align*}
max(hd(X)) &= max(X)\\
min(soc(X)) &= min(X)
\end{align*}
\end{proposition}

\begin{proof}
The radical of $A$ is a graded submodule of $A$ (see Kelarev, \cite{kelarev1992jacobson}), and so $rad(X) = Xrad(A)$ is a graded submodule of $X$.  Thus $hd(X) = X/rad(X)$ is graded with $supp(hd(X))$ $\subset supp(X)$.  Therefore $max(hd(X)) \le max(X)$.

To establish the reverse inequality, take a nonzero element $x\in X^{max(X)}$.  If $x \notin rad(X)$, then the image of $x$ in $hd(X)$ is a nonzero element in degree $max(X)$, and we are done.  Suppose $x \in rad(X)$.  Note that since $X$ is finitely generated and $A$ is finite-dimensional, $X$ is also finite dimensional.  Thus $rad^k(X)$ becomes zero for sufficiently large $k$.  Since $x$ is nonzero, there is a maximum $n> 0$ such that $x \in rad^n(X) = Xrad^n(A)$.  Write $x = \sum_{i=1}^m x_i a_i$ for some homogeneous $a_i \in rad^n(A), x_i \in X$.  Without loss of generality, we may assume that all terms are nonzero and that $deg(x_i a_i) = deg(x)$ for all $i$.  Since $deg(x) = max(X)$ and $A$ is non-positively graded, we must have that $deg(x_i) = max(X)$ and $deg(a_i) = 0$ for all $i$.  Since each $a_i \in rad^n(A)$ and $x \notin rad^{n+1}(X)$, there must be some $j$ such that $x_j \notin rad(X)$.  Thus we have obtained a nonzero $x_j \in X^{max(X)} - rad(X)$, and so $max(hd(X)) = max(X)$.

For the second equation, note that $soc(X)$ is a graded submodule of $X$ (see N\v{a}st\v{a}sescu and Van Oystaeyen, \cite{n1985note}).  Thus $supp(soc(X)) \subset supp(X)$ and so $min(soc(X)) \ge min(X)$.

For the reverse inequality, it suffices to show that $soc(X) \cap X^{min(X)} \neq 0$.  Since $A$ is non-positively graded, $X^{min(X)}A \subset X^{min(X)}$ and so $X^{min(X)}$ is a submodule of $X$.  Since $X$ is finite-dimensional, $X^{min(X)}$ has a simple submodule and thus has nonzero intersection with $soc(X)$.  Therefore $min(soc(X)) = min(X)$.
\end{proof}

\begin{proposition}
\label{loop lemma}
Let $X \in Ob(A\text{-grmod})$.  Then\\
1)  $max(\Omega X) \le max(X)$\\
2)  $min(\Omega' X) \ge min(X)$\\
3)  $max(P^{-n}_X(n)) = max(\Omega^nX(n)) \le max(X)-n$\\
4)  $min(I^n_X(-n)) = min((\Omega')^n X(-n)) \ge min(X)+n$\\
(See Section \ref{resolutions} for notation.)
\end{proposition}

\begin{proof}
Since $I \subset A^{op}\otimes_k A$, we have that $max(I) \le max(A^{op} \otimes_k A) = 0$.  Thus $max(\Omega X) =  max(X \otimes_A I) \le max(I) + max(X) \le max(X)$.

Similarly, $min(\Omega'X) = min(\Hom^\bullet_A(I, X)) \ge min(X) - max(I) \ge min(X)$.

The last two equations follow from the first two and the definitions of the standard projective and injective resolutions.
\end{proof}

Recall from Section \ref{resolutions} that the functor $\Omega(1)$ is an autoequivalence of $A\grstab$ and $A\dgstab$.  By replacing $A\grstab$ and $A\dgstab$ with equivalent categories $\widetilde{A\grstab}$ and $\widetilde{A\dgstab}$ (see Section \ref{autoequivalence}), we may assume without loss of generality that $\Omega(1)$ is an automorphism of both categories.  We let $\Omega^{-1}$ denote the inverse of $\Omega$, and we shall identify it with the isomorphic functor $\Omega'$.

Going forward, we shall write $\Omega^{-n}$ to mean $(\Omega')^n$ for $n \ge 0$, even on $A\grmod$ and $A\dgmod$.  This is a dangerous abuse of notation as $\Omega$ is not invertible in either category.  However, adopting this convention allows us to greatly simplify certain expressions and is safe as long as we avoid expressions of the form $\Omega\Omega^{-1}X$ outside the stable category.

We obtain the following corollary of Proposition \ref{loop lemma}.

\begin{proposition}
\label{finiteness lemma}
Let $X, Y \in Ob(A\text{-grmod})$.  Then $\Hom_{A\text{-grstab}}(X, \Omega^nY(n)) = 0$ for all but finitely many $n \in \mathbb{Z}$.
\end{proposition}
\begin{proof}
By Proposition \ref{loop lemma}, $max(\Omega Y(1)) \le max(Y(1)) = max(Y) -1$ and so $max(\Omega^n Y(n)) \le max(Y) - n$.  Thus for $n>>0$, we have that $max(\Omega^n Y(n)) < min(X)$.  Similarly, $min(\Omega^{-n} Y(-n)) \ge min(Y) + n$, and so for $n>>0$ we have that $min(\Omega^{-n} Y(-n)) > max(X)$.  Thus $\Hom_{A\text{-grmod}}(X, \Omega^n Y(n)) = 0$ for all but finitely many $n$.
\end{proof}

We are now ready to state the main definitions.  

\begin{definition}
\label{nice category}
Let $\mathcal{C}(A)$ be the category given by:\\
1)  $Ob(\mathcal{C}(A)) = Ob(A\grstab)$\\
2)  For $X, Y \in Ob(\mathcal{C}(A))$, $\Hom_{\mathcal{C}(A)}(X, Y) =  \bigoplus_{n \in \mathbb{Z}}\Hom_{A\grstab}(X, \Omega^nY(n))$\\
3)  For $(f_n)_{n\in \mathbb{Z}} : X \rightarrow Y$ and $(g_m)_{m\in \mathbb{Z}} : Y \rightarrow Z$, define composition by\\
$(g_m)\circ(f_n) = ( \sum_{i\in \mathbb{Z}} \Omega^i g_{j-i}(i) \circ f_i)_{j\in \mathbb{Z}}$.
\end{definition}

\begin{remark}
If we do not wish to assume that $\Omega(1)$ is an automorphism of $A\grstab$, natural isomorphisms $\varepsilon_{n,m}: \Omega^n \Omega^m \rightarrow \Omega^{n+m}$ satisfying the appropriate coherence conditions must be inserted into the composition formula.
\end{remark}

We note that the sum in the composition formula is finite by Proposition \ref{finiteness lemma}.  It is clear that $\mathcal{C}(A)$ is an additive category.  In fact, $\mathcal{C}(A)$ is precisely the orbit category $A\grstab/\Omega(1)$ as defined by Keller, \cite{Keller2005}.  Keller shows that while such a category need not be triangulated, it can always be included in a "triangulated hull".  We shall see that $A\dgstab$ is the triangulated hull of $\mathcal{C}(A)$.

\begin{proposition}
\label{krull-schmidt}
The orbit category $\mathcal{C}(A)$ is a Krull-Schmidt category.
\end{proposition}

\begin{proof}
Since $A\grstab$ is Krull-Schmidt and the natural map $A\grstab \rightarrow \mathcal{C}(A)$ is additive and essentially surjective, any object of $\mathcal{C}(A)$ can be written as a direct sum of indecomposable objects in $A\grstab$.  Thus, it suffices to show that any indecomposable object $X$ of $A\grstab$ has local endomorphism ring in $\mathcal{C}(A)$.

First, we claim that for any indecomposable $X$ in $A\grmod$ and any $i \neq 0$, any map $f: X \rightarrow \Omega^iX(i) \rightarrow X$ lies in $rad(\End_{A\grmod}(X))$.  By Proposition \ref{loop lemma}, $supp(X) \nsubseteq supp(\Omega^iX(i))$, hence $f$ cannot be surjective and thus is not an isomorphism.  Since $\End_{A\grmod}(X)$ is a local finite-dimensional algebra, $f$ lies in the unique maximal two-sided ideal, which is equal to the Jacobson radical.  Since $\End_{A\grstab}(X)$ is a quotient of $\End_{A\grmod}(X)$, we also have that the image of $f$ lies in $rad(\End_{A\grstab}(X))$.

Let $X \in A\grstab$ be indecomposable.  We must show that $\End_{\mathcal{C}(A)}(X)$ is local.  Write $V_n = \Hom_{A\grstab}(X, \Omega^n(X)(n))$ for the $n$-th graded component of $\End_{\mathcal{C}(A)}(X)$.  We claim that the subspace $V := rad(V_0) \oplus \bigoplus_{n\neq 0} V_n$ is the unique maximal two-sided ideal of $\End_{\mathcal{C}(A)}(X)$.

To show $V$ is a two-sided ideal, take $f_i$ in the $i$th graded piece of $V$ and $g_j \in V_j$.  If $i + j \neq 0$, then $g_j f_i \in V$.  If $i = -j \neq 0$, then $g_j f_i : X \rightarrow \Omega^iX(i) \rightarrow X$ is an element of $rad(V_0)$.  Finally, if $i = j = 0$, then we immediately have that $g_0 f_0 \in rad(V_0)$.  Thus $V$ is a left ideal, and a parallel argument shows it is a right ideal.

Clearly, $\End_{\mathcal{C}(A)}(X)/V \cong \End_{A\grstab}(X)/rad(\End_{A\grstab}(X))$, which is a division ring since $X$ is indecomposable.  Thus $V$ is maximal.

To show $V$ is the unique maximal ideal, it suffices to show that it is equal to $rad(\End_{\mathcal{C}(A)}(X))$.  As a maximal two-sided ideal, $V$ contains the radical.  For the reverse inclusion, it suffices to show that every element $f = (f_i) \in V$ is nilpotent, since the Jacobson radical contains every nil ideal.

Let $N$ be such that $V_i = 0$ for all $| i | > N$.  Then note that the $i$-th graded piece of $f^{n}$ is a sum of maps of the form $X \xrightarrow{h_{1}} \Omega^{i_1}X(i_1) \xrightarrow{h_2} \Omega^{i_2}X(i_2) \cdots \xrightarrow{h_{n}} \Omega^{i}X(i)$, where the $h_j$ are translations of the $f_k$ by various powers of $\Omega(1)$.  If $| i_j| > N$ for any $j$, the composite map is zero, so we may assume that $| i_j | \le N$ for all $1 \le j \le n$.  If $n > m(2N+1)$ for some $m \ge 1$, then by the pigeonhole principle there exists $-N \le r \le N$ such that $\Omega^rX(r)$ appears as the codomain of one of the $h_j$ at least $m+1$ times.  Grouping terms, we can then express the composition as $X \xrightarrow{\phi_0} \Omega^rX(r) \xrightarrow{\phi_1} \Omega^rX(r) \xrightarrow{\phi_2} \cdots \xrightarrow{\phi_m} \Omega^rX(r) \xrightarrow{\phi_{m+1}} \Omega^iX(i)$, where the $\phi_k$ are compositions of successive $h_j$.

I claim that for all $1 \le k \le m$, $\phi_k$ lies in $rad(\End_{A\grstab}(\Omega^rX(r)))$.  Note that since $f_0$ lies in the radical of $\End_{A\grstab}(X)$ and $\Omega(1)$ is an autoequivalence, we have that $\Omega^r(f)(r)$ lies in the radical of the local ring $\End_{A\grstab}(\Omega^rX(r))$.  Thus, if one of the factors of $\phi_k$ is $\Omega^r(f_0)(r)$, then we are done.  If not, then we must have that $\phi_k$ factors through $\Omega^jX(j)$ for some $j \neq r$, which again guarantees that $\phi_k$ lies in the radical.

We have shown that for $n > m(2N+1)$, each component of $f^n$ is a sum of terms of the form $\prod_{i=0}^{m+1}\phi_i$, with $\phi_i \in rad(\End_{A\grstab}(\Omega^rX(r)))$ for each $1 \le i \le m$ and some $-N \le r \le N$.  But $rad(\End_{A\grstab}(\Omega^rX(r)))$ is nilpotent for each $r$, hence $f$ is nilpotent.
\end{proof}

We now define the inclusion functor $F_A: \mathcal{C}(A) \rightarrow A\text{-dgstab}$.  The obvious choice would be for $F_A$ to act as the identity on objects and send the morphism $(f_n)_n : X \rightarrow Y$ to the sum of its components $\sum_{n\in \mathbb{Z}} \psi_{n, Y} \circ f_n$, where the $\psi_{n,Y}: \Omega^n Y(n) \rightarrow Y$ are isomorphisms chosen so that all the summands share a common domain.  However, in order for this process to be functorial, the morphisms $\psi_{n, Y}$ must be satisfy appropriate compatibility conditions.

\begin{lemma}
\label{coherence lemma}
There exists a family of natural isomorphisms $\{ \psi_{n}: \Omega^n(n)  \rightarrow id_{A\dgstab} \mid n \in \mathbb{Z}\}$ satisfying:\\
i)  $\psi_0= id_{A\dgstab}$\\
ii)  For all $n, m \in \mathbb{Z}, \psi_m \circ (\psi_n \circ\Omega^m (m)) = \psi_{n+m}$
\end{lemma}

\begin{proof}
Let $\psi_1 : \Omega(1) \rightarrow id_{A\dgstab}$ be the natural isomorphism defined in Section \ref{resolutions}.  Let $\psi_{-1} = (\psi_1 \circ \Omega^{-1}(-1) )^{-1}: \Omega^{-1}(-1) \rightarrow id_{A\dgstab}$.  Let $\psi_0 = id_{A\dgstab}$.  For $n \ge 2$, recursively define $\psi_n = \psi_1 \circ (\psi_{n-1} \circ \Omega (1))$ and analogously for $n \le -2$.  It is clear that $\{\psi_n\}$ satisfies i) and ii).
\end{proof}

\begin{remark}  
If we do not assume that $\Omega(1)$ is an autormorphism of $A\dgstab$, we must again insert appropriately chosen natural isomorphisms $\varepsilon_{n,m}: \Omega^n \Omega^m \rightarrow \Omega^{n+m}$ into condition ii).
\end{remark}

\begin{definition}
\label{inclusion functor}
Let $\psi_n: \Omega^n(n) \rightarrow id_{A\dgstab}$ be the natural isomorphisms defined in Lemma \ref{coherence lemma}.  
Let $F_A: \mathcal{C}(A) \rightarrow A\text{-dgstab}$ be the functor given by:\\
1)  $F_A$ acts as the identity on objects.\\
2)  Given $f = (f_n)_{n\in \mathbb{Z}} \in \Hom_{C(A)}(X, Y)$, let $F_A(f) = \sum_{n\in \mathbb{Z}}  \psi_{n, Y} \circ f_n$.
\end{definition}

\begin{proposition}
$F_A$ as defined above is a functor.
\end{proposition}
\begin{proof}
Take $X \in \mathcal{C}(A)$.  The identity morphism on $X$ is $(\delta_{0,n} id_X)_n$.  Therefore $F_A((\delta_{0,n} id_X)_n) = \sum_{n\in \mathbb{Z}} \psi_{n, X} \circ \delta_{0,n}id_X = \psi_{0, X} = id_X$.  Thus $F_A$ preserves identity morphisms.

Given $(f_n)_n: X \rightarrow Y, (g_m)_m: Y \rightarrow Z$ in $\mathcal{C}(A)$, we have
\begin{eqnarray*}
F_A(g_m)\circ F_A(f_n) & = & (\sum_m \psi_{m, Z} \circ g_m)\circ (\sum_n \psi_{n, Y} \circ f_n)\\
& = & \sum_{m, n} \psi_{m, Z} \circ g_m \circ \psi_{n, Y} \circ f_n\\
& = & \sum_{m, n} \psi_{m, Z} \circ \psi_{n, \Omega^mZ(m)} \circ \Omega^n g_m(n) \circ f_n\\
& = & \sum_{m, n} \psi_{m+n, Z} \circ \Omega^n g_m(n) \circ f_n\\
& = & \sum_{j = m+n} \psi_{j, Z}\circ (\sum_{n} \Omega^{n} g_{j-n}(n) \circ f_n)\\
& = & F_A((g_m) \circ (f_n))
\end{eqnarray*}
\end{proof}


\subsection{Embedding $\mathcal{C}(A)$ into $A\dgstab$}

We now state the main theorem.

\begin{theorem}
\label{dgstable}
$F_A: \mathcal{C}(A) \rightarrow A\text{-dgstab}$ is fully faithful, and the image of $F_A$ generates $A\text{-dgstab}$ as a triangulated category.
\end{theorem}

We prove the theorem with a sequence of lemmas below.

\begin{definition}
Let $X, Y \in A\grmod$, viewed as dg-modules with zero differential.  If $X$ and $Y$ are nonzero, let $N= N_{X,Y} := max\{n \le 0 \mid max(\Omega^{-n}Y(-n)) < min(X)\}$.  Define the \textbf{bridge complex from $X$ to $Y$} to be the complex $R^\bullet_{X,Y} = B^{\ge N}_Y$ (see Sections \ref{complexes} and \ref{resolutions} for notation) if $X$ and $Y$ are both nonzero, and $R^\bullet_{X,Y} = 0$ otherwise.
\end{definition}

By Proposition \ref{loop lemma}, $N_{X,Y}$ is well-defined.  We will omit the subscript when it is clear from context.

By unwinding the definitions, we obtain a quasi-isomorphism of complexes $\Omega^{-N}(Y)[-N] \hookrightarrow R^\bullet_{X,Y}$.  In particular, $H^N(R^\bullet_{X,Y}) \cong \Omega^{-N}(Y)$ and $H^k(R^\bullet_{X,Y}) = 0$ for $k \neq N$.  We also note that $ker(d^n_{R_{X,Y}})(-n) = \Omega^{-n}X(-n)$ for all $n \ge N$.

Morphisms in $A\dgstab$ can be represented as equivalence classes of roofs $X\xrightarrow{f} M \xleftarrow{s} Y$, where $s$ has perfect cone.  The primary challenge in understanding morphisms in $A\dgstab$ is that perfect dg-modules need not arise from complexes of graded projective modules.  However, by restricting our attention to dg-modules with zero differential, we can bypass this difficulty by using the bridge complexes defined above.

\begin{lemma}
\label{roof lemma}
Let $X, Y \in A\grmod$.  Then any morphism in $\Hom_{A\dgstab}(X, Y)$ can be expressed as a roof of the form
\begin{eqnarray*}
\begin{tikzcd}
X \arrow[dr, swap, "f"]& &Y \arrow[dl, "i"]\\
& \widehat{\tau_{\le 0}R_{X,Y}} &
\end{tikzcd}
\end{eqnarray*}
where $i$ is induced by the natural map $Y \rightarrow R^\bullet_{X,Y}$.
\end{lemma}

\begin{proof}
If $X$ or $Y$ is zero, then the result is immediate, so assume neither $X$ nor $Y$ is zero. 
Any morphism $X \rightarrow Y$ in $A\dgstab$ can be represented as a roof
\begin{eqnarray*}
\begin{tikzcd}
X \arrow[dr, swap, "g"] & &Y \arrow[dl, "s"]\\
& M &
\end{tikzcd}
\end{eqnarray*}
where $M \in A\dgmod, g, s \in \text{Mor}(D^b_{dg}(A))$, and there is an exact triangle $P \xrightarrow{\alpha}Y \xrightarrow{s} M \xrightarrow{} P(1)$ in $D^b_{dg}(A)$, with $P \in D^{perf}_{dg}(A)$.  By changing $P$ up to quasi-isomorphism, we may assume without loss of generality that $P$ is strictly perfect.

Let $p_n$ denote the natural map of complexes $P^{\ge n}_Y \hookrightarrow P^\bullet_Y \twoheadrightarrow Y$ and let $i_n: Y \hookrightarrow C(p_n)$ denote the natural inclusion of complexes.  If $n \ge 1$, note that $p_n$ is the map from the zero complex to $Y$ and $i_n$ is the identity map on $Y$.  Note also that $C(p_{N+1}) = \tau_{\le 0}R^\bullet_{X,Y}$ and $\widehat{i_{N+1}} = i$.

We first show that every morphism can be expressed as a roof of the form
\begin{eqnarray*}
\begin{tikzcd}
X \arrow[dr]& &Y \arrow[dl, "\widehat{i_{k}}"]\\
& \widehat{C(p_{k})} &
\end{tikzcd}
\end{eqnarray*}
for some $k \le N+1$.

Since $p: P^\bullet_Y \twoheadrightarrow Y$ is a quasi-isomorphism of complexes, we obtain a morphism $\widehat{p}^{-1} \circ \alpha: P \rightarrow \widehat{P_Y}$ in $D_{dg}(A)$.  Since $P \in D^{perf}_{dg}(A)$, the underlying graded module of $P$ is finitely generated.  Thus $supp(P)$ is bounded.  Note that $max(P^{-k}_Y) = max(\Omega^kY)$, thus by Proposition \ref{loop lemma} the sequence $\{max(P^{-k}_Y(k)\}_k$ is strictly decreasing.  Then we may choose $k <<0$ such that $k \le N+1$ and $max(P^{k-1}_Y(-k+1)) < min(P)$.  Then the short exact sequence of dg-modules $0 \rightarrow \widehat{P^{\ge k}_Y} \hookrightarrow \widehat{P_Y} \twoheadrightarrow \widehat{P^{<k}_Y} \rightarrow 0$ yields an exact triangle $\widehat{P^{\ge k}_Y} \rightarrow \widehat{P_Y} \rightarrow \widehat{P^{<k}_Y} \rightarrow \widehat{P^{\ge k}_Y}(1)$ in $D_{dg}(A)$.  Since
\begin{equation*}
max(\widehat{P^{<k}}) = max(P^{k-1}_Y(-k+1)) < min(P)
\end{equation*}
we have that $\Hom_{Ho_{dg}(A)}(P, \widehat{P^{<k}_Y}) = 0$.  Since $P$ is strictly perfect, morphisms in the derived and homotopy categories coincide, and so $\Hom_{D_{dg}(A)}(P, \widehat{P^{<k}_Y}) = 0$.

We obtain a morphism of triangles in $D_{dg}(A)$:
\begin{eqnarray*}
\begin{tikzcd}
P \arrow[r, "id"] \arrow[d, "h", dotted] & P \arrow[r] \arrow[d, "\widehat{p}^{-1} \circ \alpha"] &  0 \arrow[r] \arrow[d] & P(1) \arrow[d, "h(1)", dotted]\\
\widehat{P^{\ge k}_Y} \arrow[r] & \widehat{P_Y} \arrow[r] &\widehat{P^{< k}_Y} \arrow[r] & \widehat{P^{\ge k}_Y}(1) 
\end{tikzcd}
\end{eqnarray*}
Postcomposing the left square with $\widehat{P_Y} \xrightarrow{\widehat{p}} X$, we obtain $\alpha = \widehat{p_k} \circ h$.  We obtain a morphism of triangles in $D^b_{dg}(A)$:
\begin{eqnarray*}
\begin{tikzcd}
P \arrow[r, "\alpha"] \arrow[d, "h"] & Y \arrow[r, "s"] \arrow[d, "id"] & M \arrow[r] \arrow[d, "g' ", dotted] & P(1) \arrow[d, "h(1)"]\\
\widehat{P^{\ge k}_Y} \arrow[r, "\widehat{p_k}"] & Y \arrow[r, "\widehat{i_k}"] & \widehat{C(p_k)} \arrow[r] & \widehat{P^{\le k}_Y}(1)
\end{tikzcd}
\end{eqnarray*}
Since $\widehat{P^{\ge{k}}_Y} \in D^{perf}_{dg}(A)$, the roof $\widehat{i_k}^{-1} \circ (g' \circ g): X \rightarrow Y$ defines a morphism in $A\dgstab$.  It follows immediately from the above diagram that the roofs $s^{-1}\circ g$ and $\widehat{i_k}^{-1} \circ (g' \circ g)$ are equivalent in $A\dgstab$.

It remains to show that $k$ can be replaced by $N+1$.  Since $k \le N+1$ by definition, we have an exact sequence of dg-modules $0 \rightarrow \widehat{C(p_{N+1})} \hookrightarrow \widehat{C(p_k)} \twoheadrightarrow \widehat{(P_Y^{\ge k})^{\le N}}(1) \rightarrow 0$ arising from the underlying exact sequence of complexes.  We also have that 
\begin{eqnarray*}
max(\widehat{(P_Y^{\ge k})^{\le N}}(1)) &=& max(P_Y^{N}(-N)(1))\\
 & = & max(\Omega^{-N} Y(-N+1))\\
 & < & min(X)
\end{eqnarray*}
The last inequality is true by definition of $N$.  Thus $\Hom_{Ho^b_{dg}(A)}(X, \widehat{(P_Y^{\ge k})^{\le N}}(1)) = 0$ and, since $\widehat{(P_Y^{\ge k})^{\le N}}(1)$ is strictly perfect, $\Hom_{D^b_{dg}(A)}(X, \widehat{(P_Y^{\ge k})^{\le N}}(1)) =  0$.  We obtain a morphism of triangles in $D^b_{dg}(A)$:
\begin{eqnarray*}
\begin{tikzcd}
X \arrow[r, "id"] \arrow[d, "f ", dotted] & X \arrow[r] \arrow[d, "g \circ g' "] & 0 \arrow[r] \arrow[d] & X(1) \arrow[d, "f(1)", dotted]\\
\widehat{C(p_{N+1})} \arrow[r, hook] & \widehat{C(p_k)} \arrow[r] & \widehat{(P_Y^{\ge k})^{\le N}}(1) \arrow[r] & \widehat{C(p_N)}(1)
\end{tikzcd}
\end{eqnarray*}
We also have that $\widehat{i_k}$ factors as $Y \xrightarrow{\widehat{i_{N+1}}} \widehat{C(p_{N+1})} \hookrightarrow \widehat{C(p_k)}$.  It follows that the roof $\widehat{i_{N+1}}^{-1} \circ f$ defines a morphism in $A\dgstab$ which is equal to $\widehat{i_k}^{-1} \circ (g' \circ g)$.  Since $\widehat{C(p_N)} = \widehat{\tau_{\le 0} R_{X,Y}}$ and $\widehat{i_{N+1}} = i$, we are done.
\end{proof}

Having found a convenient choice of roofs between $X$ and $Y$, we now investigate maps between $X$ and $\widehat{\tau_{\le 0}R_{X,Y}}$ in the derived category.  This investigation shall yield a method for computing morphisms between zero-differential modules.

\begin{lemma}
\label{computation lemma}
Let $X, Y\in A\grmod$.  Then we have an isomorphism
\begin{eqnarray*}
\xi: \Hom_{Ho^+_{dg}(A)}(X, \widehat{R_{X,Y}})& \xrightarrow{\sim}& \Hom_{A\dgstab}(X, Y)\\
f &\mapsto & 
\begin{tikzcd}
X \arrow[dr, swap, "\phi^{-1}\circ f"] & & Y \arrow[dl, "i"]\\
& \widehat{\tau_{\le 0}R_{X,Y}} &
\end{tikzcd}
\end{eqnarray*}
where $\phi$ is the natural inclusion of $\widehat{\tau_{\le 0}R_{X,Y}}$ into $\widehat{R_{X,Y}}$.
\end{lemma}

\begin{proof}
Let $f \in \Hom_{Ho^+_{dg}(A)}(X, \widehat{R_{X,Y}})$.  In order for $\xi$ to be well-defined, we must show that $\phi^{-1} \circ f \in Mor(D_{dg}(A))$ can be represented by a roof in $D^b_{dg}(A)$.  By Proposition \ref{loop lemma}, the sequence $\{min(\widehat{\tau_{>M}R_{X,Y}})\}_M$ strictly increases with $M$.  Since $X$ is finitely generated, there exists $M>>0$ such that the image of $f$ lies in $\widehat{\tau_{\le M}(R_{X,Y})}$.  It is clear that the inclusion $\phi$ also factors through $\widehat{\tau_{\le M}(R_{X,Y})}$, and the inclusion of $\widehat{\tau_{\le M}(R_{X,Y})}$ into $\widehat{R_{X,Y}}$ is a quasi-isomorphism.  Thus $R_{X,Y}$ can be replaced by the bounded dg-module $\widehat{\tau_{\le M}(R_{X,Y})}$ in the roof $\phi^{-1} \circ f$, and so we may view $\phi^{-1} \circ f$ as a morphism in $Mor(D^b_{dg}(A))$.  Thus $\xi$ is well-defined.

We now prove surjectivity of $\xi$.  Since $R_{X,Y} \in Ho^+(A\grproj)$, we have that $\Hom_{Ho_{dg}(A)}(X, \widehat{R_{X,Y}}) \cong \Hom_{D_{dg}(A)}(X, \widehat{R_{X,Y}})$.  Post-composition with $\phi^{-1}$ yields an isomorphism:
\begin{eqnarray*}
\Hom_{Ho^+_{dg}(A)}(X, \widehat{R_{X,Y}})& \xrightarrow{\sim}& \Hom_{D^b_{dg}(A)}(X, \widehat{\tau_{\le 0}R_{X,Y}})\\
f& \mapsto& \phi^{-1} \circ f
\end{eqnarray*}
It follows immediately from Lemma \ref{roof lemma} that the map
\begin{eqnarray*}
\Hom_{D^b_{dg}(A)}(X, \widehat{\tau_{\le 0}R_{X,Y}})& \twoheadrightarrow& \Hom_{A\dgstab}(X, Y)\\
g &\mapsto & 
\begin{tikzcd}
X \arrow[dr, swap, "g"] & & Y \arrow[dl, "i"]\\
& \widehat{\tau_{\le 0}R_{X,Y}} &
\end{tikzcd}
\end{eqnarray*}
is surjective.  The composition of these two maps is precisely $\xi$, which is therefore surjective.

It remains to show injectivity.  Suppose that $\xi(f) = 0.$  Then there exists a morphism $s: \widehat{\tau_{\le 0} R_{X,Y}} \rightarrow M$ in $D^b_{dg}(A)$ such that $s \circ \phi^{-1} \circ f = 0$ and $C(s)$ is strictly perfect.  We obtain a morphism of triangles in $D^b_{dg}(A)$:

\begin{eqnarray*}
\begin{tikzcd}
X \arrow[r, "id"] \arrow[d, "\alpha", dotted] & X \arrow[r] \arrow[d, "\phi^{-1} \circ f"] & 0 \arrow[r] \arrow[d] & X(1)\arrow[d, "\alpha(1)", dotted]\\
C(s)(-1) \arrow[r, "\beta"] & \widehat{\tau_{\le 0} R_{X,Y}} \arrow[r, "s"] & M \arrow[r] & C(s)
\end{tikzcd}
\end{eqnarray*}
Since $C(s)(-1)$ is strictly perfect, we can choose to represent $\alpha$ by a morphism in $Ho^b_{dg}(A)$.  From the above diagram and the fact that $R^\bullet_{X,Y} \in Comp^+(A\grproj)$ it follows that $f = \phi \circ \beta \circ \alpha$ in $Ho^b_{dg}(A)$.

Note that the natural inclusion of complexes $\epsilon: \Omega^{-N} Y[-N] \hookrightarrow R_{X,Y}$ is a quasi-isomorphism, hence so is $\widehat{\epsilon}: \Omega^{-N} Y(-N) \hookrightarrow \widehat{R_{X,Y}}$.  Since $C(s)(-1)$ is strictly perfect, the roof $\widehat{\epsilon}^{-1} \circ \phi \circ \beta: C(s)(-1) \rightarrow \Omega^{-N}Y(-N)$ can be represented by a morphism $\gamma$ in $Ho^b_{dg}(A)$.  Repeating the argument of the previous paragraph, we have that $\phi \circ \beta = \widehat{\epsilon} \circ \gamma$ in $Ho^b_{dg}(A)$.

By the above two paragraphs, we have $f = \phi \circ \beta \circ \alpha = \widehat{\epsilon} \circ \gamma \circ \alpha$ in $Ho^b_{dg}(A)$.  But this means that $f$ factors through $\Omega^{-N}Y(-N)$ in the homotopy category.  By definition of $N$, $max(\Omega^{-N}Y(-N)) < min(X)$ and so $\Hom_{Ho^b_{dg}}(X, \Omega^{-N}Y(-N)) = 0$.  Thus $f = 0$ and $\xi$ is injective.
\end{proof}

In the next three lemmas, we relate morphisms in $\mathcal{C}(A)$ to those in $A\dgstab$ via the homotopy category.

\begin{lemma}
\label{morphism lemma 1}
Let $X\in A\grmod$.  Let $(P^\bullet, d^\bullet_P)\in Comp(A\grmod)$.  Suppose that $\Hom_{A\grmod}(X, ker(d^n_P)(-n)) = 0$ for almost all $n$.  Let $i_n: ker(d^n_P)(-n) \hookrightarrow \widehat{P}$ denote the inclusion (of dg-modules).  Then the map
\begin{eqnarray*}
\Phi: \bigoplus_{n\in \mathbb{Z}}\Hom_{A\grmod}(X, ker(d^n_P)(-n)) & \rightarrow & \Hom_{A\dgMod}(X, \widehat{P})\\
(f_n)_n &\mapsto & \sum_n i_n \circ f_n
\end{eqnarray*}
is an isomorphism of vector spaces.
\end{lemma}

\begin{proof}
By hypothesis, the sum in the definition of $\Phi$ is finite, so $\Phi$ is a well-defined $k$-linear map.  It remains to construct $\Phi^{-1}$.  Given $f \in \Hom_{A\dgMod}(X, \widehat{P})$, we have $d_{\widehat{P}}\circ f = f \circ d_X =0$, since $X$ has zero differential.  Thus $im(f) \subset ker(d_{\widehat{P}}) = \bigoplus_n ker(d^n_P)(-n)$.  Let $\pi_n$ denote the projection onto the $n$th summand, and define $\Phi^{-1}(f) = (\pi_n\circ f)_n$; it is easy to verify that $\Phi^{-1}$ is inverse to $\Phi$.
\end{proof}

\begin{lemma}
\label{morphism lemma 2}
Let all notation and assumptions be as in Lemma \ref{morphism lemma 1}.  Assume in addition that $P^\bullet \in Comp(A\grproj)$ and that $P^\bullet$ is exact at each $n$ for which $\Hom_{A\grmod}(X, ker(d^n_P)(-n))$ is nonzero.  Then $\Phi$ induces an isomorphism:
\begin{eqnarray*}
\overline{\Phi}: \bigoplus_{n\in \mathbb{Z}}\Hom_{A\grstab}(X, ker(d^n_P)(-n)) & \rightarrow & \Hom_{Ho_{dg}(A)}(X, \widehat{P})\\
(f_n)_n &\mapsto & \sum_n i_n\circ f_n
\end{eqnarray*}
\end{lemma}

\begin{proof}
Take $(f_n)_n \in \bigoplus_{n\in \mathbb{Z}}\Hom_{A\grmod}(X, ker(d^n_P)(-n))$.  By Lemma \ref{morphism lemma 1}, it suffices to show that $\Phi(f_n)$ is nullhomotopic if and only if $f_n$ factors through a projective module for all $n$.  We also note that $d_{\widehat{P}}$ is $A$-linear, since $d_A = 0$.

Suppose that $\Phi(f_n)$ is nullhomotopic and fix $k \in \mathbb{Z}$.  Let $h: X \rightarrow \widehat{P}(-1)$ be a homotopy.  Since $d_X = 0$, we have that $\Phi(f_n) = d_{\widehat{P}}(-1) \circ h$ (as morphisms of graded modules).  As a graded module, $\widehat{P} = \bigoplus_n P^n(-n)$; let $\pi_n$ be the projection onto the $n$th summand.  From the proof of Lemma \ref{morphism lemma 1}, we have that $f_k = \pi_k \circ \Phi(f_n)$, and so $f_k = \pi_k \circ d_{\widehat{P}}(-1)\circ h$.  Thus $f_k$ factors through the graded projective module $\widehat{P}(-1)$.

Now suppose that for each $n$, $f_n$ factors as $X \xrightarrow{a_n} Q_n \xrightarrow{b_n} ker(d^n_P)(-n)$ for some $Q_n \in A\grproj$.  We shall define a nullhomotopy of $\Phi(f_n)$ by constructing maps $h_n: X \rightarrow P^{n-1}(-n)$.  If $f_n = 0$, let $h_n = 0$.  If $f_n$ is nonzero, then $P^\bullet$ is exact at $n$, and so $P^{n-1}(-n)$ surjects onto $ker(d^n_P)(-n)$ via the differential.  Since $Q_n$ is projective, $b_n$ lifts to $c_n: Q_n \rightarrow P^{n-1}(-n)$.  Define $h_n = c_n \circ a_n$, as summarized by the diagram below.
\begin{eqnarray*}
\begin{tikzcd}
X \arrow[r, "a_n"] \arrow[d, "h_n", dotted] & Q_n \arrow[d, "b_n"] \arrow[dl, swap, "c_n", dotted]\\
P^{n-1}(-n) \arrow[r, swap, "d^{n-1}_P(-n)", two heads] & ker(d^n_P)(-n)
\end{tikzcd}
\end{eqnarray*}

Viewing $P^{n-1}(-n)$ as a graded submodule of $\widehat{P}(-1)$, define $h := \sum_n h_n : X \rightarrow \widehat{P}(-1)$.  Since all but finitely many of the $h_n$ are zero, $h$ is a well-defined morphism of graded modules.  It is easy to check that $\Phi(f_n) = d_{\widehat{P}}(-1) \circ h$, hence $\Phi(f_n)$ is nullhomotopic.
\end{proof}

\begin{lemma}
\label{morphism lemma 3}
Let $X, Y\in A\grmod$.  Then there is an isomorphism
\begin{eqnarray*}
\chi:  \Hom_{\mathcal{C}(A)}(X, Y) & \xrightarrow{\sim} & \Hom_{Ho_{dg}(A)}(X, \widehat{R_{X, Y}})\\
(f_n) & \mapsto & \sum_{n}i_n \circ f_{-n}
\end{eqnarray*}
where $i_n : \Omega^{-n}Y(-n) = ker(d^n_{R_{X,Y}})(-n) \hookrightarrow \widehat{R_{X,Y}}$ is the natural inclusion of dg-modules for $n \ge N$ and the zero map for $n < N$.
\end{lemma}

\begin{proof}
We may assume that $X$ and $Y$ are nonzero.  We first show that the hypotheses of Lemmas \ref{morphism lemma 1} and \ref{morphism lemma 2} are satisfied by $X$ and $R_{X,Y}$.  By Proposition \ref{loop lemma} we have that $\Hom_{A\grmod}(X, ker(d^n_{R_{X,Y}})(-n)) = 0$ for all but finitely many $n$.  By construction, $R_{X,Y}\in Comp(A\grproj)$ is exact at all $n \neq N$.  By the definition of $N$, $\Hom_{A{\dgmod}}(X, \Omega^{-N}Y(-N)) = 0$.  Thus the hypotheses of Lemmas \ref{morphism lemma 1} and \ref{morphism lemma 2} are satisfied.

$\chi$ is precisely the composition of the isomorphism
\begin{eqnarray*}
\Hom_{\mathcal{C}(A)}(X, Y) & \xrightarrow{\sim} & \prod_{n\in \mathbb{Z}}\Hom_{A\grstab}(X, \Omega^{-n}Y(-n))\\
(f_n) & \mapsto & (f_{-n})
\end{eqnarray*}
with $\overline{\Phi}$ of Lemma \ref{morphism lemma 2}.  Thus $\chi$ is an isomorphism.
\end{proof}

We are now ready to prove Theorem \ref{dgstable}.

\begin{lemma}
\label{fully faithful}
The functor $F_A$ is fully faithful.
\end{lemma}

\begin{proof}
Let $X, Y \in A\grmod$.  From Lemmas \ref{computation lemma} and \ref{morphism lemma 3}, we obtain an isomorphism $\xi \circ \chi: \Hom_{\mathcal{C}(A)}(X, Y) \xrightarrow{\sim} \Hom_{A\dgstab}(X, Y)$.  It remains to show that this isomorphism is induced by $F_A$.

We have that $\xi \circ \chi(f_n) = \sum_{n} i^{-1}\circ \phi^{-1} \circ i_n \circ f_{-n}$.  Since $f_{-n} = 0$ for $n \le N$, it suffices to show that $i^{-1}\circ \phi^{-1} \circ i_n = \psi_{-n, Y}$ for $n > N$, where the $\psi_{-n,Y}$ are defined in Lemma \ref{coherence lemma}.

It follows easily from the definitions that $\psi_{-n, Y}$ can be represented by roofs of the form
\begin{eqnarray*}
\begin{tikzcd}
\Omega^{-n}Y(-n) \arrow[dr]& & Y \arrow[dl]\\
& \widehat{\tau_{\le 0}B_Y^{\ge n}} & 
\end{tikzcd}
& \text{ for } N < n \le 0\\
\begin{tikzcd}
\Omega^{-n}Y(-n) \arrow[dr]& & Y \arrow[dl]\\
& \widehat{\tau_{\le n}B^{\ge 0}_Y} & 
\end{tikzcd}
& \text{ for } n \ge 0
\end{eqnarray*}
where all morphisms are inclusions of dg-modules and have either acyclic or perfect cones.  We then obtain commutative diagrams of inclusions:
\begin{eqnarray*}
\begin{tikzcd}
\Omega^{-n}(Y)(-n)\arrow[dr] \arrow[dddr, "i_n", swap, bend right] & & Y\arrow[dl] \arrow[ddl, "i", bend left]\\
& \tau_{\le 0}B^{\ge n}_Y \arrow[d]&\\
& \tau_{\le 0} B^{\ge N}_Y \arrow[d, "\phi"] &\\
& B^{\ge N}_Y
\end{tikzcd}
& \text{ for } N < n \le 0
\end{eqnarray*}
\begin{eqnarray*}
\begin{tikzcd}
\Omega^{-n}Y(-n) \arrow[dr] \arrow [dddr, "i_n", swap, bend right] & & Y \arrow[dl] \arrow[d, "i"]\\
& \tau_{\le n}B^{\ge 0}_Y \arrow[d] & \tau_{\le 0}B^{\ge N}_Y \arrow[dl] \arrow[ddl, "\phi", bend left]\\
& \tau_{\le n} B^{\ge N}_Y \arrow[d] &\\
& B^{\ge N} & &
\end{tikzcd}
& \text{ for } n \ge 0
\end{eqnarray*}
Every map in the above diagrams is either a quasi-isomorphism or has perfect cone.  (This is immediate for all maps except $i_n$.  It then follows that $i_n$ is an isomorphism in $A\dgstab$, hence has perfect cone.)  Thus the above diagrams show that the roof defining $\psi_{n, Y}$ is equivalent to $i^{-1} \circ \phi^{-1} \circ i_n$ for all $n > N$.  
\end{proof}

\begin{lemma}
\label{generate}
The image of $F_A$ generates $A\dgstab$ as a triangulated category.
\end{lemma}

\begin{proof}
Let $\mathcal{T}$ denote the full, replete, triangulated subcategory of $A\dgstab$ generated by the image of $F_A$.  Let $M\in A\dgstab$.  Since $ker(d_M)$ is a zero-differential dg-submodule of $M$, it lies in $\mathcal{T}$.  Thus if $ker(d_M) = M$, we are done.  Likewise, if $ker(d_M) = 0$, then $M$ is acyclic and thus isomorphic to zero in $A\dgstab$.  Thus we may assume that $ker(d_M)$ is a proper nontrivial dg-submodule of $M$.

The short exact sequence $0 \rightarrow ker(d_M) \hookrightarrow M \twoheadrightarrow M/ker(d_M) \rightarrow 0$ induces an exact triangle in $A\dgstab$.  $M$ is finitely generated, hence finite-dimensional, thus $dim(M/ker(d_M)) < dim(M)$.  By induction on $dim(M)$, we may assume that $M/ker(d_M)) \in \mathcal{T}$.  Since $ker(d_M)$, $M/ker(d_M) \in \mathcal{T}$ it follows that $M \in \mathcal{T}$.  Thus $\mathcal{T} = A\dgstab$.
\end{proof}


\subsection{The Dg-stable Category as a Triangulated Hull}

Theorem \ref{dgstable} suggests that $A\dgstab$ the triangulated hull of $\mathcal{C}(A)$.  To establish this, we must express our constructions in the language of dg-categories.

Given a category $\mathcal{C}$ with an automorphism $\Phi$, we can define a category $\mathcal{C}_{gr}$, enriched in $k\grmod$, with the same objects as $\mathcal{C}$ and morphisms $\Hom_{\mathcal{C}_{gr}}^n(X,Y) :=\Hom_{\mathcal{C}}(X, \Phi^n(Y))$.  The grading shift functor $(1)$ is an automorphism of the categories $A\grmod$, $A\grstab$, and $\mathcal{C}(A)$, allowing us to define the corresponding graded categories, which we view as dg-categories.  The natural functors between these categories are all compatible with grading shifts and hence lift to functors of dg-categories.  Likewise, the functor $\Omega$ commutes with grading shifts and thus defines a dg-endofunctor of each of these categories.

We can view the differential graded algebra $A$ as a dg-category with one object, and we denote the associated dg-category of (bounded) dg-modules by $A\dgmod_{dg}$.  We have that $Z^0(A\dgmod_{dg}) \cong A\dgmod$ and $H^0(A\dgmod_{dg}) \cong Ho^b_{dg}(A)$.  There is a natural dg-functor $A\grmod_{gr} \rightarrow A\dgstab_{dg}$.

Define the dg-derived category $\mathcal{D}^b_{dg}(A)$ to be the dg-quotient of $A\dgmod_{dg}$ by the full dg-subcategory of acyclic dg-modules.  Similarly, let $A\dgstab_{dg}$ denote the dg-quotient of $A\dgmod_{dg}$ by the full dg-subcategory of objects which are quasi-isomorphic to a perfect dg-module.  We have that $H^0(\mathcal{D}^b_{dg}(A)) \cong D^b_{dg}(A)$, $H^0(A\dgstab_{dg}) \cong A\dgstab$, and there are natural dg-functors $A\dgmod_{dg} \rightarrow \mathcal{D}^b_{dg}(A) \rightarrow A\dgstab_{dg}$.  

Since the functor $\Omega$ is given by tensoring with the bimodule $I$ defined in Section \ref{resolutions}, it induces a dg-functor $A\dgmod_{dg} \rightarrow A\dgmod_{dg}$.  Since this functor preserves acyclic and perfect dg-modules, by the universal property of the dg-quotient, it descends to a dg-endofunctor of $\mathcal{D}^b_{dg}(A)$ and $A\dgstab_{dg}$.

Since the natural dg-functor $A\grmod_{gr} \rightarrow A\dgstab_{dg}$ sends projective modules to zero, we obtain an induced dg-functor $A\grstab_{gr} \rightarrow A\dgstab_{dg}$ which is the identity on objects.  We would like this functor to descend to $\mathcal{C}(A)_{gr}$.

\begin{proposition}
$\mathcal{C}(A)_{gr}$ is the dg-orbit category of $A\grstab_{gr}$ by the functor $\Omega(1)$.
\end{proposition}
\begin{proof}
By construction, the projection map $A\grstab \rightarrow \mathcal{C}(A)$ is essentially surjective, and the natural map
\begin{equation*}
\bigoplus_{c\in \mathbb{Z}} colim_{r>>0} \Hom_{A\grstab_{gr}}(\Omega^rX(r), \Omega^{r+c}Y(r+c)) \rightarrow \Hom_{\mathcal{C}(A)_{gr}}(X, Y)
\end{equation*}
is an isomorphism of dg $k$-modules.  The result then follows by Keller's characterization of the orbit category (\cite{Keller2005}, Section 9.3, part d) of the Theorem).
\end{proof}

Write $\mathcal{B} = A\grstab_{gr}$ and $\mathcal{C} = A\dgstab_{dg}$.  Let $F_{dg}: \mathcal{B} \rightarrow \mathcal{C}$ denote the natural dg-functor.  Since $F_{dg}$ is the identity on objects, we can identify it with the dg $\mathcal{B}^{op} \otimes \mathcal{C}$-module $(X, Y) \mapsto \Hom_{\mathcal{C}}(Y, X \otimes_A A)$.  Similarly, we can identify $F_{dg}\circ \Omega(1)$ with the dg bimodule given by $(X,Y) \mapsto \Hom_{\mathcal{C}}(Y, X \otimes_A I(1))$.  By construction of $I$, we have a closed morphism $\phi:  A \rightarrow I(1)$ in $\mathcal{D}^b_{dg}(A^{op}\otimes_k A)$ which has perfect cone in $D^b_{dg}(A^{op}\otimes_k A)$.  Thus, $\phi$ descends to a closed morphism in $\mathcal{C}$ which becomes an isomorphism in $H^0(\mathcal{C}) = A\dgstab$.  Thus $\phi$ induces a quasi-isomorphism of dg-bimodules from $F_{dg}$ to $F_{dg} \circ \Omega(1)$.  Thus, by the universal property of the orbit category, $F_{dg}$ descends to $\mathcal{C}(A)_{dg}$, and it is clear that $H^0(F_{dg})$ is the functor $F_A$ defined in \ref{inclusion functor}.

\begin{corollary}
\label{dgstable strong}
$A\dgstab$ is equivalent to the triangulated hull $Tr(\mathcal{C}(A)_{gr})$.
\end{corollary}
\begin{proof}
By the universal property of the pretriangulated hull, the dg-functor $F_{dg}: \mathcal{C}(A)_{gr} \rightarrow A\dgstab_{dg} \hookrightarrow PreTr(A\dgstab_{dg})$ factors through a dg-functor $\widehat{F}: PreTr(\mathcal{C}(A)_{gr}) \rightarrow PreTr(A\dgstab_{dg})$.  By construction, $H^0(A\dgstab_{dg}) \cong A\dgstab$ generates $H^0(PreTr(A\dgstab_{dg})) = Tr(A\dgstab_{dg})$ as a triangulated category.  Since $A\dgstab$ is already triangulated, we have that $A\dgstab \cong Tr(A\dgstab_{dg})$.

The functor $H^0(\widehat{F}): Tr(\mathcal{C}(A)_{gr}) \rightarrow A\dgstab$ is exact and restricts to $F_A$ on the image of $\mathcal{C}(A)$ inside $Tr(\mathcal{C}(A)_{gr})$.  The image of $F_A$ generates $A\dgstab$ as a triangulated category, and the image of $H^0(\widehat{F})$ is triangulated, hence $H^0(\widehat{F})$ is essentially surjective.  Furthermore, $H^0(\widehat{F})$ is fully faithful on $\mathcal{C}(A)$.  Since $\mathcal{C}(A)$ generates $Tr(\mathcal{C}(A))$ as a triangulated category, it follows by a standard argument that $H^0(\widehat{F})$ is fully faithful.  Thus $A\dgstab$ is equivalent to $Tr(\mathcal{C}(A)_{gr})$.
\end{proof}

Having proven Corollary \ref{dgstable strong}, we make a few brief remarks on when two graded algebras $A$ and $B$ have equivalent dg-stable categories.  If $D^b(A\grmod)$ is equivalent to $D^b(B\grmod)$, then the equivalence can be expressed as tensoring by a tilting complex.  (See Rickard, \cite{rickard1989morita}).  This functor is still defined on the derived category of dg-modules and remains an equivalence.  Furthermore, this equivalence preserves the perfect dg-modules and thus induces an equivalence between the dgstable categories.  Thus, graded derived equivalence implies dg-stable equivalence.  However, we can say more:

\begin{lemma}
\label{gr to dgstable equivalence}
Let $A$ and $B$ be finite-dimensional self-injective algebras, graded in non-positive degree.  Suppose there is an equivalence of triangulated categories $G:  A\grstab \rightarrow B\grstab$ which commutes with grading shifts.  Then $G$ induces an equivalence between $A\dgstab$ and $B\dgstab$.
\end{lemma}

\begin{proof}
Since $G$ is a triangulated equivalence, it commutes with $\Omega$.  Thus $G$ commutes with the functor $\Omega(1)$ and induces a functor $\mathcal{C}(A) \rightarrow \mathcal{C}(B)$.  Given $Y \in B\grstab$, there exists $X\in A\grstab$ such that $G(X) \cong Y$ in $B\grstab$, hence in $\mathcal{C}(B)$.  Thus the induced functor on the orbit category is essentially surjective.  Given $X,Y \in A\grstab$, the map $\Hom_{\mathcal{C}(A)}(X,Y) \rightarrow \Hom_{\mathcal{C}(B)}(G(X), G(Y))$ is bijective, since $\Hom_{A\grstab}(X, \Omega^n(Y)(n)) \rightarrow \Hom_{B\grstab}(G(X), \Omega^nG(Y)(n))$ is bijective for each $n\in \mathbb{Z}$.  Thus $G: \mathcal{C}(A) \rightarrow \mathcal{C}(B)$ is an equivalence.  Similarly, $G$ induces an equivalence of dg-categories $\mathcal{C}(A)_{gr} \rightarrow \mathcal{C}(B)_{gr}$.

The composition $\mathcal{C}(A)_{gr} \rightarrow \mathcal{C}(B)_{gr} \rightarrow PreTr(\mathcal{C}(B)_{gr})$ factors through the pretriangulated hull $PreTr(\mathcal{C}(A)_{gr})$.  Applying $H^0$, we obtain an exact functor $\overline{G}: A\dgstab \rightarrow B\dgstab$ extending $G$.  Since $G$ is an equivalence, it follows that $\overline{G}$ is an equivalence.
\end{proof}



\section{Essential Surjectivity}
\label{Essential Surjectivity}

\subsection{Morphisms Concentrated in One Degree}

Let $A$ be a non-positively graded finite-dimensional self-injective algebra over a field $k$.  Let $F_A: \mathcal{C}(A) \rightarrow A\dgstab$ be the functor of Definition \ref{inclusion functor}.  Having shown in the previous section that $F_A$ is fully faithful, we now investigate conditions on $A$ that guarantee essential surjectivity.

Since the image of $\mathcal{C}(A)$ generates $A\dgstab$ as a triangulated category, $F_A$ is essentially surjective if and only if the essential image $Im(F_A)$ is a triangulated subcategory of $A \dgstab$, if and only if $\mathcal{C}(A)$ admits a triangulated structure compatible with $F_A$.  In general, this is not the case.

The primary obstacle to essential surjectivity is that there is no natural candidate for the cone of a morphism $(f_n)_n: X \rightarrow Y$ for which more than one $f_n$ is nonzero.  The cone of such a morphism will correspond to a dg-module that does not arise from a chain complex and need not be isomorphic to a chain complex modulo projectives.

However, by imposing restrictions on the algebra $A$, we can prevent this scenario from occurring.  In this case, we shall see that $F_A$ is essentially surjective, hence an equivalence.

\begin{definition}
\label{niceness}
Let $X, Y \in Im(F_A) \subset A\dgstab$.  We say that a morphism $f: X \rightarrow Y$ is \textbf{chain-like} if $C(f) \in Im(F_A)$.  We say that $X \in Im(F_A)$ is \textbf{nice} if $f: X \rightarrow Y$ is chain-like for all $Y \in Im(F_A)$ and all $f \in \Hom_{A\dgstab}(X,Y)$.
\end{definition}

Note that $Im(F_A)$ is closed under cones (and thus triangulated) if and only if all of its objects are nice.  In fact, it suffices for all the indecomposable objects of $Im(F_A)$ to be nice:

\begin{lemma}
\label{direct sums preserve niceness}
Let $X_1, X_2 \in Im(F_A)$ be nice.  Then $X_1 \oplus X_2$ is nice.
\end{lemma}

\begin{proof}
Take $Y\in Im(F_A)$ and $(f_1 \, f_2 ): X_1 \oplus X_2 \rightarrow Y$.  Applying the octahedron axiom to the composition $f_1 = (f_1 \, f_2 ) \circ i_1$, we obtain the following diagram:
\begin{eqnarray*}
\begin{tikzcd}
X_1 \arrow[rr, "i_1"] \arrow[rrrr, bend left, "f_1"] &  & X_1 \oplus X_2 \arrow[rr, " (f_1 \, f_2) "] \arrow[dl] & & Y \arrow[dl]\\
& X_2 \arrow[ul, "(1)"] \arrow[dr, dotted, "g"] & & C(f_1 \, f_2) \arrow[ul, "(1)"] \arrow[ll, "(1)"] &\\
& & C(f_1) \arrow[ur, dotted, "h"] & &
\end{tikzcd}
\end{eqnarray*}

The bottom-most triangle is exact, so $C(f_1 \, f_2)$ is the cone of $g: X_2 \rightarrow C(f_1)$.  Since $X_1$ is nice and $Y\in Im(F_A)$, it follows that $C(f_1) \in Im(F_A)$.  Since $X_2$ is nice, $C(f_1 \, f_2) \in Im(F_A)$.  Thus $X_1 \oplus X_2$ is nice.
\end{proof}

The following condition is sufficient to guarantee that all indecomposables are nice.

\begin{lemma}[One Morphism Rule]
\label{one morphism rule}
Suppose for all indecomposable $X, Y \in A\grmod$, $\Hom_{A\grstab}(X, \Omega^nY(n)) \neq 0$ for at most one $n\in \mathbb{Z}$.  Then every indecomposable object of $Im(F_A)$ is nice.  In particular, $F_A: \mathcal{C}(A) \rightarrow A\dgstab$ is an equivalence.
\end{lemma}

\begin{proof}
Let $X \in A\grmod$ be indecomposable.  Let $M = \bigoplus_{i = 1}^n Y_i \in Im(F_A)$, with $Y_i$ indecomposable.  Changing each $Y_i$ up to isomorphism, we may assume without loss of generality that $\Hom_{A\grstab}(X, \Omega^nY_i(n)) = 0$ for $n \neq 0$.  Then any morphism $(f_n)_n:  X\rightarrow M$ in $\mathcal{C}(A)$ is concentrated in degree $0$ and thus can be identified with the morphism $f_0$ in $A\grstab$.  Since $F_A$ is fully faithful, any morphism $f: X \rightarrow M$ in $A\dgstab$ can be represented by a morphism in $A\grmod$.  

Choosing a monomorphism $i: X \hookrightarrow I$, where $I$ is injective, we obtain a short exact sequence of graded $A$-modules $0 \rightarrow X \xrightarrow{(f \, i)}  M \oplus I \rightarrow C \rightarrow 0$ which induces an exact triangle in $D^b_{dg}(A)$, hence in $A\dgstab$.  Since $I \cong 0$ in $A\dgstab$, this triangle is equivalent to one of the form $X \xrightarrow{f} M \rightarrow C \rightarrow X(1)$.  $C$ is a cone of $f$ and lies in the image of $F_A$ (since it is in $A\grmod$).  Thus $X$ is nice.

The second statement follows immediately from Lemma \ref{direct sums preserve niceness} and the preceding remarks.
\end{proof}

\begin{remark}  The hypotheses of Lemma \ref{one morphism rule} are quite restrictive.  However, we note that if $A$ is concentrated in degree 0 (that is, ungraded), then the One Morphism Rule is trivially satisfied.

In this case, any indecomposable object $X\in A\grmod$ is concentrated in a single degree $n$, and so $\Omega^nX(n)$ is concentrated in degree $0$.  Thus every object of $\mathcal{C}(A)$ is isomorphic to an object concentrated in degree zero, and $\Hom_{\mathcal{C}(A)}(X, Y)$ $\cong \Hom_{A\stab}(X,Y)$ for any two such objects $X$ and $Y$.  Thus $\mathcal{C}(A)$ is equivalent to $A\stab$.

Furthermore, a dg-module over $A$ is the same as a complex of $A$-modules.  Thus $A\dgstab = D^b_{dg}(A)/D^{perf}_{dg}(A) = D^b(A\mmod)/D^{perf}(A\mmod)$.  Thus, in the case where $A$ an ungraded finite-dimensional, self-injective algebra, Theorem \ref{dgstable} and Lemma \ref{one morphism rule} precisely yield Rickard's Theorem \cite{rickard1989derived} that $A\stab \cong D^b(A\mmod)/D^{perf}(A\mmod)$.  Thus it is appropriate to view $\mathcal{C}(A)$ as the differential graded analogue of the additive definition of the stable module category.
\end{remark}

\subsection{Nakayama Algebras}

\begin{definition}
A \textbf{Nakayama algebra} is a finite-dimensional algebra for which all indecomposable projective and injective modules are uniserial.
\end{definition}

Since every indecomposable module has an indecomposable projective cover, it follows that every indecomposable module over a Nakayama algebra is uniserial.

\begin{proposition}
\label{uniserial loop lemma}
Let $A$ be a finite-dimensional, self-injective Nakayama algebra, graded in non-positive degree.  Let $X \in A\grmod$ be indecomposable and not projective.  Let $p_X: P_X \twoheadrightarrow X$ be a projective cover of $X$ and let $i_X: X \hookrightarrow I_X$ be an injective hull of $X$.  Let $K = ker(p_X)$ and $C = coker(i_X)$.  Then $max(K) \le min(X)$, and $max(X) \le min(C)$.
\end{proposition}

\begin{proof}
For any $k\ge 0$ and $Y \in A\grmod$, let $L^k(Y) = rad^k(Y)/rad^{k+1}(Y)$ be the $k$-th radical layer of $Y$.  Let $l(Y)$ denote the length of $Y$.  If $Y$ is indecomposable, then it is uniserial and so $L^k(Y)$ is simple for $0 \le k < l(Y)$.

Since $X$ is indecomposable, $P_X$ is indecomposable, hence uniserial, and we have that $K = ker(p_X) = rad^{l(X)}(P_X)$ and $X \cong P_X/rad^{l(X)}(P_X)$.  Let $M = rad^{l(X)-1}(P_X)/rad^{l(X)+1}(P_X)$.  Then $hd(M) = L^{l(X)-1}(P_X) = soc(X)$ and $soc(M) = L^{l(X)}(P_X) = hd(K)$ are simple.  Thus,
\begin{align*}
max(K) &= max(hd(K))
 = max(soc(M))
 = min(soc(M))
 = min(M)\\
 &\le max(M)
 = max(hd(M))
 = max(soc(X))
 = min(soc(X))\\
 &= min(X)
\end{align*}

The proof of the second inequality is precisely dual, using the socle layers of $I_X$.
\end{proof}

\begin{lemma}
\label{uniserial}
Let $A$ be a non-positively graded, finite-dimensional, self-injective Naka\-yama algebra.  Then the conditions of Lemma \ref{one morphism rule} are satisfied.  In particular, $F_A$ is an equivalence.
\end{lemma}

\begin{proof}
Let $X, Y \in A\grmod$ be indecomposable, and suppose that there is a nonzero morphism $f: X \rightarrow \Omega^m(Y)(m)$ in $A\grstab$ for some $m \in \mathbb{Z}$.  Changing $Y$ up to isomorphism in $Im(F_A) \subset A\dgstab$, we may assume that $m = 0$.  Then there is a nonzero morphism from $X$ to $Y$ in $A\grmod$, and so $max(hd(X)) \ge min(soc(Y))$.

Note that $\Omega(Y) \in A\grmod$ has a unique (up to isomorphism) non-projective direct summand $K$, which is the kernel of a projective cover of $Y$.  Then $\Omega(Y) \cong K$ in $A\grstab$, hence in $Im(F_A)$.  Identifying $\Omega(Y)$ with $K$, Proposition \ref{uniserial loop lemma} states that $max(\Omega(Y)(1)) < min(soc(Y)) \le max(hd(X)) = min(hd(X))$.  Thus $\Hom_{A\grstab}(X, \Omega Y(1)) = 0$ and, by induction, $\Hom_{A\grstab}(X, \Omega^n Y(n)) = 0$ for all $n > 0$.  A dual argument shows that $max(hd(X)) < min(\Omega^{-1}Y(-1))$ and so $\Hom_{A\grstab}(X, \Omega^n Y(n)) = 0$ for all $n < 0$.
\end{proof}


\subsection{An Example of the Failure of Essential Surjectivity}

Let $A =k[x,y]/(x^2, y^2)$, where $k = \mathbb{C}$.  We grade $A$ by putting $x$ in degree $0$ and $y$ in degree $-1$.  It is easy to check that $A$ is symmetric, hence self-injective.  Up to grading shift, $A$ has a single simple graded module, $S$, which has dimension one and upon which both $x$ and $y$ act by zero.  Therefore, up to grading shift, the only indecomposable projective module is $A$ itself.

The representation theory of $A$ is closely related to that of the Kronecker quiver, 
\begin{eqnarray*}
\begin{tikzcd}
v_1 \arrow[r, bend left, "a"] \arrow[r, bend right, swap, "b"] & v_2
\end{tikzcd}
\end{eqnarray*}
We let $B$ denote the path algebra of this quiver, with $a$ in degree 0 and $b$ in degree $-1$.  $B$ has two simple modules $S_1$ and $S_2$, one corresponding to each vertex.  There is a one-to-one correspondence between the indecomposable graded $A$-modules, excluding the projective module, and the graded $B$-modules, excluding the simple module $S_2$. (See Chapter 4.3 of \cite{benson1991representations} for the ungraded case.  Note that the graded case follows from the same argument.)

The classification of graded indecomposable $B$-modules is known.  (For instance, see Seidel \cite{seidel2004exact}, Section 4.)  Transferring these results to $A$-modules, we obtain the following classification of the indecomposable graded $A$-modules.  Up to shift, these are:\\
1)  The indecomposable projective module, $A$.\\
2)  For $n \ge 0$, the module $K^n$, which is of dimension $2n + 1$.  As a graded vector space, $K^n = V \oplus W$, where $V = \bigoplus^n_{i = 0} k(i)$, $W = \bigoplus^n_{i = 1} k(i)$, and $x$ and $y$ act by mapping $V$ into $W$ via the matrices
$\begin{pmatrix}
0		& 1	 	& 		& 0\\
\vdots	& 	 	& \ddots	& \\
0		& 0		&		&1
\end{pmatrix}$
and
$\begin{pmatrix}
1	 & 		& 0	& 0\\
 	 & \ddots	&	& \vdots\\
0	&		& 1	& 0
\end{pmatrix}$
, respectively.  Note that in $A\grstab$ we have that $K^n \cong \Omega^nS$ for all $n\ge 0$.  We shall use the notation $\Omega^n S$ going forward.\\
3)  For $n < 0$, the module $K^n$, which is of dimension $2n + 1$.  As a graded vector space, $K^n = V \oplus W$, where $V = \bigoplus^n_{i = 1} k(-i)$, $W = \bigoplus^n_{i = 0} k(-i)$, and $x$ and $y$ act by mapping $V$ into $W$ via the matrices
$\begin{pmatrix}
0		& \dots 	& 0\\
1		& 		& 0\\
 	 	& \ddots	& \\
 0		&		&1
\end{pmatrix}$
and
$\begin{pmatrix}
1		& 		& 0\\
 	 	& \ddots	& \\
 0		&		&1\\
 0		& \dots 	& 0
\end{pmatrix}$
, respectively.  Once again, we note that $K^n \cong \Omega^nS$ in $A\grstab$ for all $n < 0$.  We shall use the notation $\Omega^n S$ going forward.\\
4)  For $n > 0$, the module $M_{0, n}$, which is of dimension $2n$.  As a graded vector space, $M_{0, n} = V \oplus W$, where $V =  \bigoplus^{n}_{i = 1} k(-i)$, $W = \bigoplus^{n-1}_{i = 0} k(-i)$, and $x$ and $y$ act by mapping $V$ into $W$ via the $n\times n$ matrices
$\begin{pmatrix}
0		& \dots 	&		& 0\\
1		& 		& 0 		& \\
 	 	& \ddots	& 		& \vdots\\
 0		&		&1		& 0
\end{pmatrix}$
and
$\begin{pmatrix}
1	&  		& 0\\
 	& \ddots	& \\
 0	&		& 1
\end{pmatrix}$
, respectively.\\
5)  For $n > 0$, the module $M_{\infty, n}$, which is of dimension $2n$.  As a graded vector space, $M_{\infty, n} = V \oplus W$, where $V = W =  \bigoplus^{n-1}_{i = 0} k(-i)$, and $x$ and $y$ act by mapping $V$ into $W$ via the $n\times n$ matrices
$\begin{pmatrix}
1	&  		& 0\\
 	& \ddots	& \\
 0	&		& 1
\end{pmatrix}$
and
$\begin{pmatrix}
0		& 1 		&		& 0\\
\vdots	& 		& \ddots 	& \\
 	 	& 0		& 		& 1\\
 0		&		& \dots	& 0
\end{pmatrix}$
, respectively.

Note that for any of the modules described in 2-5 above, $hd(X) \cong V$ and $soc(X) \cong W$ as graded modules, each with $x$ and $y$ acting by $0$.

The following computations are straightforward; we leave them to the reader.  Below, $n\ge 0$ and $m \ge 1$.\\
$dim\Hom_{A\grstab}(S, \Omega^mS(k)) = 
\begin{cases}
1 & -m \le k \le -1\\
0 & o.w.
\end{cases}$
\\
$dim\Hom_{A\grstab}(S, \Omega^{-n}S(k)) = 
\begin{cases}
1 & 0 \le k \le n\\
0 & o.w.
\end{cases}$
\\
$dim\Hom_{A\grstab}(S, M_{0, m}(k)) = 
\begin{cases}
1 & 0 \le k \le m-1 \\
0 & o.w.
\end{cases}$
\\
$dim\Hom_{A\grstab}(S, M_{\infty, m}(k)) = 
\begin{cases}
1 & 0 \le k \le m-1 \\
0 & o.w.
\end{cases}$
\\
$dim\Hom_{A\grstab}(M_{0, m}, S(k)) = 
\begin{cases}
1 & -m \le k \le -1 \\
0 & o.w.
\end{cases}$
\\
$dim\Hom_{A\grstab}(M_{\infty, m}, S(k)) = 
\begin{cases}
1 & -m +1 \le k \le 0 \\
0 & o.w.
\end{cases}$
\\
$dim\Hom_{A\grstab}(M_{\infty, m}, M_{\infty, 1}(k)) = 
\begin{cases}
1 & k = 0, -m +1 \\
0 & o.w.
\end{cases}$
\\
$dim\Hom_{A\grstab}(M_{0, m}, M_{0, 1}(k)) = 
\begin{cases}
1 & k = 0, -m\\
0 & o.w.
\end{cases}$
\\

In $\mathcal{C}(A)$, functors $\Omega$ and $(-1)$ are isomorphic, so our list of indecomposable objects shrinks.  In $A\grstab$, note that $\Omega M_{0, n} = M_{0,n}$ and $\Omega M_{\infty, n} = M_{\infty, n}(1)$; thus $M_{0, n} \cong M_{0, n}(1)$ and $M_{\infty, n} \cong M_{\infty, n}(2)$ in $\mathcal{C}(A)$.  Thus, a complete list of indecomposable objects in $\mathcal{C}(A)$ up to isomorphism is:\\
1)  $S(n)$, for $n \in \mathbb{Z}$.\\
2)  $M_{0, m}$, for $m > 0$.\\
3)  $M_{\infty, m}(n)$, for $m > 0$ and $n \in \{0, 1\}$.

The sizes of the following Hom sets in $A\dgstab$ are an immediate consequence of the above computations for $A\grstab$ and some simple counting arguments.\\
$dim\Hom_{A\dgstab}(S, S(n)) = 
\begin{cases}
\lfloor \frac{n}{2} \rfloor + 1 & n \ge 0\\
\lfloor \frac{|n|}{2} \rfloor & n < 0
\end{cases}$
\\
$dim\Hom_{A\dgstab}(S, M_{0, m}) = m$
\\
$dim\Hom_{A\dgstab}(S, M_{\infty, m}(n)) =
\begin{cases}
\lfloor \frac{m+1}{2} \rfloor  & n \equiv 0 \mod 2\\
\lfloor \frac{m}{2} \rfloor & n \equiv 1 \mod 2
\end{cases}$
\\
$dim\Hom_{A\dgstab}(M_{0, m}, S(n)) = m$
\\
$dim\Hom_{A\dgstab}(M_{0, m}, M_{0, 1}) = 2$
\\
$dim\Hom_{A\dgstab}(M_{\infty, m}(r), S(n)) = 
\begin{cases}
\lfloor \frac{m+1}{2} \rfloor  & n - r \equiv 0 \mod 2\\
\lfloor \frac{m}{2} \rfloor & n - r \equiv 1 \mod 2
\end{cases}$
\\
$dim\Hom_{A\dgstab}( M_{\infty, m}, M_{\infty, 1}) =
\begin{cases}
1 & m \equiv 0 \mod 2\\
2  & m \equiv 1 \mod 2
\end{cases}$
\\
$dim\Hom_{A\dgstab}(M_{\infty, m}, M_{\infty, 1}(1)) =
\begin{cases}
1 & m \equiv 0 \mod 2\\
0  & m \equiv 1 \mod 2
\end{cases}$

We are now ready to construct an object $K$ of $A\dgstab$ lying outside of $\mathcal{C}(A)$.  From the above computations, we have that $dim\Hom_{A\dgstab}(S, S(2)) = 2$; for a basis we can take the unique (up to a nonzero scalar) morphisms $f_{-1}: S \rightarrow \Omega^{-1}(S)(1) \cong S(2)$ and $f_{-2}: S \rightarrow \Omega^{-2}S \cong S(2)$.  Let $g = f_{-1} + f_{-2}$, and let $K$ be the cone of $g$ in $A\dgstab$.  We shall show that $K$ does not lie in the image of $F_A$.

\begin{proposition}
\label{counterexample pt 1}
$dim\Hom_{A\dgstab}(K, S(n)) = 1$ for all $n \ge 3$
\end{proposition}

\begin{proof}

Consider the triangle $S \xrightarrow{g} S(2) \rightarrow K \rightarrow S(1)$ which defines $K$.  Choosing some $n \ge 2$, we apply $\Hom_{A\dgstab}(-, S(n))$ and observe the resulting long exact sequence.  We will show that $g(-k)^*: \Hom_{A\dgstab}(S(2-k), S(n)) \rightarrow$ $\Hom_{A\dgstab}(S(-k), S(n))$ is injective for all $k \ge 0$.  From this, it will follow from the long exact sequence that
\begin{align*}
dim\Hom_{A\dgstab}(K(-k-1), S(n)) &= dim\Hom_{A\dgstab}(S(-k), S(n))\\
 &- dim\Hom_{A\dgstab}(S(2-k), S(n))\\
 &= 1
\end{align*}
for all $k \ge 0$, and we will have $dim\Hom_{A\dgstab}(K, S(n)) = 1$ for all $n \ge 3$.

Applying the functor $(k)$, it suffices to show that $g^*: \Hom_{A\dgstab}(S(2), S(r))$ $\rightarrow$ $\Hom_{A\dgstab}(S, S(r))$ is injective for all $r \ge 2$, where $r = n+k$.  Interpreting $f_{-1}$ and $f_{-2}$ as morphisms in $A\dgstab$, we have that $g^* = f^*_{-1} + f^*_{-2}$.  If we are given a nonzero morphism $h_s: S(2) \rightarrow \Omega^s S(r+s)$ in $A\grstab$, a straightforward computation shows that both $\Omega^{-1}h_s(-1) \circ f_{-1}: S \rightarrow \Omega^{s-1}S(r+s-1)$ and $\Omega^{-2}h_s(-2) \circ f_{-2}: S \rightarrow$ $\Omega^{s-2}S(r+s-2)$ are nonzero morphisms in $A\grstab$.  It follows immediately that $f^*_{-1}$ and $f^*_{-2}$ are injective.

We now show that $g^*$ is injective.  Let $(h_s)_s : S(2) \rightarrow S(r)$ in $\mathcal{C}(A)$.  Note that $h_s$ can be nonzero only when $-r + 2 \le s \le -\lceil \frac{r}{2} \rceil +1$.  Therefore $g^*(h_s)_s = (a_s)_s$, where\\
\begin{equation*}
a_s =
\begin{cases}
\Omega^{-2}h_{-r + 2}(-2) \circ f_{-2}& \text{ if } s = -r\\
\Omega^{-1}h_{s+1}(-1) \circ f_{-1} + \Omega^{-2}h_{s+2}(-2) \circ f_{-2} & \text{ if } -r < s < -\lceil \frac{r}{2} \rceil\\
\Omega^{-1}h_{-\lceil \frac{r}{2} \rceil +1}(-1) \circ f_{-1}& \text{ if } s = -\lceil \frac{r}{2}\rceil\\
0& \text{ otherwise }
\end{cases}
\end{equation*}

Now suppose that $g^*(h_s) = 0$.  If $(h_s)_s \neq 0$, let $N$ be the maximum $s$ such that $h_s$ is nonzero.  By injectivity of $f^*_{-1}$, we have that $N < -\lceil \frac{r}{2} \rceil + 1$, and by injectivity of $f^*_{-2}$, we have that $N > -r +2$. But then
\begin{align*}
0 = a_{N-1} &= \Omega^{-1}h_N(-1) \circ f_{-1} + \Omega^{-2} h_{N+1}(-2) \circ f_{-2}\\
 &= \Omega^{-1}h_N(-1) \circ f_{-1} + 0
\end{align*}

Injectivity of $f^*_{-1}$ implies that $\Omega^{-1}h_N(-1) = 0$, hence $h_N = 0$.  As this contradicts the definition of $N$, we must have that $h_s = 0$ for all $s$, and so $g^*$ is injective for all $r \ge 2$.  Thus $dim\Hom_{A\dgstab}(K, S(n)) = 1$ for all $n \ge 3$.
\end{proof}

Proposition \ref{counterexample pt 1} and the above computations of Hom spaces show that $K$ cannot be isomorphic to any object of $\mathcal{C}(A)$ except possibly $M_{\infty, 2}$, $M_{\infty, 2}(1)$, $M_{\infty, 1} \oplus M_{\infty, 1}(1)$, or $M_{0, 1}$.

\begin{proposition}
\label{counterexample pt 2}
$\Hom_{A\dgstab}(K, M_{\infty, 1}(k)) = 0$ for all $k$.
\end{proposition}

\begin{proof}
Again consider the triangle $S \xrightarrow{g} S(2) \rightarrow K \rightarrow S(1)$ defining $K$ and write $g =  f_{-1} + f_{-2}$.  Applying $\Hom_{A\dgstab}(-, M_{\infty, 1})$, we again show that $g^*(k): \Hom_{A\dgstab}(S(2+k), M_{\infty, 1}) \rightarrow \Hom_{A\dgstab}(S(k), M_{\infty, 1})$ is an isomorphism for all $k$.  As in Proposition \ref{counterexample pt 1}, we shall apply $(-k)$ and work instead with $g^*: \Hom_{A\dgstab}(S(2), M_{\infty, 1}(-k)) \rightarrow$ $\Hom_{A\dgstab}(S, M_{\infty, 1}(-k))$.  Since $M_{\infty, 1} \cong M_{\infty, 1}(2)$, it suffices to consider the cases $k = 0$ and $k = 1$.

If $k = 1$, both spaces are zero, and the result is immediate.  If $k = 0$, both spaces are one-dimensional, so it is enough to show that $g^*$ is not the zero map.  The unique morphism $S(2) \rightarrow M_{\infty, 1}$ is (up to rescaling) of the form $h_1: S(2) \rightarrow \Omega M_{\infty, 1}(1)$.  Then $g^*(h_1) = (r_n)_n$, where\\
\begin{equation*}
r_n = \begin{cases}
\Omega^{-1}h_1 (-1) \circ f_{-1}& \text{ if } n = 0\\
\Omega^{-2}h_1 (-2) \circ f_{-2}& \text{ if } n = -1\\
0 & \text{ otherwise}
\end{cases}
\end{equation*}

A simple computation in $A\grstab$ shows that $r_{-1} = 0$ and and $r_0$ is a nonzero element of $\Hom_{A\grstab}(S, M_{\infty, 1})$, whence $g^*$ is nonzero.  Thus $g^*$ is an isomorphism in all cases, and so $\Hom_{A\dgstab}(K, M_{\infty, 1}(k)) = 0$ for all $k$.
\end{proof}

Proposition \ref{counterexample pt 2} eliminates all remaining possibilities for $K$ except for $M_{0,1}$.  This final possibility can be eliminated by proving:

\begin{proposition}
\label{counterexample pt 3}
$\Hom_{A\dgstab}(K, M_{0, 1}) = 0$
\end{proposition}

\begin{proof}
Once again, we consider the triangle $S \xrightarrow{g} S(2) \rightarrow K \rightarrow S(1)$ defining $K$ and write $g =  f_{-1} + f_{-2}$.  We show $g^*(k): \Hom_{A\dgstab}(S(2+k), M_{0, 1}) \rightarrow \Hom_{A\dgstab}(S(k), M_{0, 1})$ is an isomorphism for all $k$.  Applying $(-k)$ and using the identity $M_{0,1} \cong M_{0,1}(1)$, we show that $g^*: \Hom_{A\dgstab}(S(2), M_{0, 1}) \rightarrow \Hom_{A\dgstab}(S, M_{0, 1})$ is an isomorphism.  Since both spaces are one-dimensional, it suffices to show that the map is nonzero.

The generator of $\Hom_{A\dgstab}(S(2), M_{0, 1})$ is $h_2 : S(2) \rightarrow \Omega^2 M_{0,1}(2)$, and so $g^*(h_2) = (r_n)_n$, where
\begin{equation*}
r_n = \begin{cases}
\Omega^{-1}h_2(-1) \circ f_{-1}& \text{ if } n =1\\
\Omega^{-2}h_2(-2) \circ f_{-2}& \text{ if } n =0\\
0 & \text{ otherwise}
\end{cases}
\end{equation*}

A straightforward computation shows that $r_1 = 0$ and $r_0$ is the generator of $\Hom_{A\dgstab}(S, M_{0,1})$.  Thus $g^*$ is an isomorphism and $\Hom_{A\dgstab}(K, M_{0,1}) = 0$.
\end{proof}

\begin{corollary}
$K$ does not lie in the image of $F_A$.  In particular, $F_A$ is not essentially surjective.
\end{corollary}


\section{Brauer Tree Algebras}
\label{Brauer Tree Algebras}

In this section we shall prove that the functor $F_A$ of Theorem \ref{dgstable} is an equivalence whenever the algebra $A$ is any non-positively graded Brauer tree algebra.  We shall work over an algebraically closed field $k$.

A \textbf{Brauer tree} consists of the data $\Gamma = (T, e, v, m)$, where $T$ is a tree, $e$ is the number of edges of $T$, $v$ is a vertex of $T$, called the \textbf{exceptional vertex}, and $m$ is a positive integer, called the \textbf{multiplicity} of $v$.  To any Brauer tree $\Gamma$, we can associate a basic finite-dimensional symmetric algebra $A_\Gamma$.  For the details of this process, we refer to \cite{schroll2018brauer}.

An important special case is $S = (S, n, v, m)$, the star with $n$ edges and exceptional vertex at the center.  In this case, the algebra $A_{S}$ is a Nakayama algebra whose indecomposable projective modules have length $nm+1$.

The following theorems are due to Bogdanic:
\begin{theorem}{(Bogdanic, \cite{bogdanic2010graded}, Theorem 4.3 and Lemma 4.9)}
\label{grading transferance}
Let $S$ be the star with $n$ vertices and multiplicity $m$.  Let $A_S$ be graded so that $soc(A_S)$ is in degree $nm$.  Let $\Gamma$ be any Brauer tree with $n$ vertices and multiplicity $m$.  Then $A_{\Gamma}$ admits a non-negative grading such that $soc(A_{\Gamma})$ is in degree $nm$, and there is an equivalence $D^b(A_S\grmod)\rightarrow D^b(A_{\Gamma}\grmod)$.
\end{theorem}

\begin{theorem}{(Bogdanic, \cite{bogdanic2010graded}, Section 11)} 
\label{Brauer tree gradings}
Let $\Gamma$ be a Brauer tree with multiplicity $m$.  Then, up to graded Morita equivalence and rescaling, $A_{\Gamma}$ possesses a unique grading.  The socle of $A_{\Gamma}$ lies in degree $dm$ for some $d \in \mathbb{Z}$, and the grading is determined up to graded Morita equivalence by $d$.  If $d > 0$, the grading can be chosen to be non-negative, and if $d<0$ the grading can be chosen to be non-positive.
\end{theorem}

From these facts, we obtain the following result:
\begin{corollary}
\label{Brauer tree stable equivalence}
Let $\Gamma$ be a Brauer tree.  Let $S$ be the star with the same multiplicity and number of edges.  Let $A_{\Gamma}$ and $A_S$ be graded so that $soc(A_{\Gamma})$ and $soc(A_S)$ lie in degree $d$ for some $d\in \mathbb{Z}$.  Then we have equivalences of triangulated categories $D^b(A_S \grmod) \rightarrow D^b(A_{\Gamma}\grmod)$ and $A_S\grstab \rightarrow A_{\Gamma}\grstab$. 
\end{corollary}

\begin{theorem}
\label{Brauer tree surjectivity}
Let $\Gamma = (T, e, v, m)$ be a Brauer graph, and let $A_\Gamma$ be non-positively graded with socle in degree $-d \le 0$.  Then $F_{A_\Gamma}$ is an equivalence.
\end{theorem}

\begin{proof}
If $\Gamma$ is the star, then $A_\Gamma$ is a Nakayama algebra and the result follows immediately from Lemma \ref{uniserial}.  If $\Gamma$ is not the star, let $S$ denote the star with the same number of edges and multiplicity as $\Gamma$.  By Theorem \ref{Brauer tree gradings} there is a nonpositive grading on $A_S$ such that $soc(A_S)$ is in degree $-d$.  Then by Theorem \ref{grading transferance}, $D^b(A_\Gamma\grmod)$ and $D^b(A_S\grmod)$ are equivalent as triangulated categories.  By a theorem of Rickard \cite{rickard1989derived}, this induces a triangulated equivalence $G: A_\Gamma\grstab \rightarrow A_S\grstab$ which commutes with grading shifts.  By Proposition \ref{gr to dgstable equivalence}, $G$ induces an equivalence between $\mathcal{C}(A_\Gamma)$ and $\mathcal{C}(A_S)$.  Since $A_S$ is a Nakayama algebra, it satisfies the hypotheses of Lemma \ref{one morphism rule}, hence $A_\Gamma$ does as well.  Thus $F_{A_\Gamma}$ is an equivalence.
\end{proof}

\begin{corollary}
Let $\Gamma$ be a Brauer tree, and let $S$ be the star with the same multiplicity and number of edges.  Let $A_{\Gamma}$ and $A_S$ both be graded with socle in degree $-d \le 0$.  Then $A_\Gamma \dgstab$ and $A_S \dgstab$ are equivalent as triangulated categories.
\end{corollary}

\begin{proof}
This follows from the use of Proposition \ref{gr to dgstable equivalence} in the previous theorem.
\end{proof}


\section{The Dg-Stable Category of the Star with $n$ Vertices}
\label{The Dg-Stable Category of the Star with n Vertices}
For $n \ge 2, d\ge 0$, let $A = A_{n, d}$ denote the graded Brauer tree algebra, with socle in degree $-d$, corresponding to the star $S$ with $n$ edges and exceptional vertex of multiplicity one.  This specifies $A$ up to graded Morita equivalence; we will choose a specific grading once we have adopted some more notation in the section below.  By the results of Section \ref{Brauer Tree Algebras}, $A\dgstab$ is equivalent to $\mathcal{C}(A)$.  We shall identify the two categories throughout this section.


\subsection{Notation, Indexing, and Grading}
We index the edges of $S$ by the set $\mathbb{Z}/n\mathbb{Z} = \{\overline{1}, \cdots, \overline{n}\}$, according to their cyclic order around the center vertex.  We define a total order $\le$ on $\mathbb{Z}/n\mathbb{Z}$ by $\overline{1} < \overline{2} < \cdots < \overline{n}$.  This order is of course not compatible with the group operation on $\mathbb{Z}/n\mathbb{Z}$.

If $P$ is a statement with a truth value, we define $\delta_P$ to be $1$ if $P$ is true and $0$ if $P$ is false.

For $x, y \in \mathbb{Z}$, define $\langle x , y \rangle$ to be the closed arc of the unit circle starting at $e^{\frac{2\pi\sqrt{-1}}{n}x}$ and proceeding counterclockwise to $e^{\frac{2\pi\sqrt{-1}}{n}y}$.  Thus $\langle x, x \rangle$ denotes a point, rather than the full circle.

With these definitions, the Ext-quiver of $A$ is a directed cycle, $C$, of length $n$.  $C$ has vertices $e_{\overline{i}}$ and edges
$\begin{tikzcd}
e_{\overline{i}} \arrow[r, "a_{\overline{i}}"] & e_{\overline{i}+\overline{1}}
\end{tikzcd}$
for all $\overline{i}\in \mathbb{Z}/n\mathbb{Z}$.  $A$ is isomorphic to, and will be identified with, the quotient of the path algebra of $C$ by the ideal generated by paths of length $n+1$.  Changing $A$ up to graded Morita equivalence, we determine the grading on $A$ by defining $deg(a_{\overline{i}}) = -d \delta_{\overline{i}=\overline{n}}$.  We denote by $S_{\overline{i}}$ the simple $A$-module corresponding to $e_{\overline{i}}$, in degree 0.  We denote by $P_{\overline{i}}$ the indecomposable projective module with head $S_{\overline{i}}$ and socle $S_{\overline{i}}(d)$.

The indecomposable $A$-modules are uniserial and determined, up to isomorphism, by their head and socle.  For $\overline{i}, \overline{j} \in \mathbb{Z}/n\mathbb{Z}$, let $M^{\overline{i}}_{\overline{j}}$ denote the indecomposable module with head $S_{\overline{i}}$ and socle $S_{\overline{j}}(d \delta_{\overline{j} < \overline{i}})$.  More specifically, for $1 \le i, j \le n$, we define $M^{\overline{i}}_{\overline{j}}$ to be the module $e_{\overline{i}}A/ e_{\overline{i}}J^l$, where $J$ is the Jacobson radical of $A$ and $l = \delta_{i>j}n + 1 +j - i$ is the length of $M^{\overline{i}}_{\overline{j}}$.  The non-projective indecomposable objects of $A\grmod$, up to grading shifts and isomorphism, are precisely $M^{\overline{i}}_{\overline{j}}$ for $\overline{i}, \overline{j} \in \mathbb{Z}/n\mathbb{Z}$.

Even when working in $A\grstab$, it will be helpful to define the "length" of $M^{\overline{i}}_{\overline{j}}$, for $1 \le i, j \le n$, to be $l(M^{\overline{i}}_{\overline{j}}) = \delta_{i>j}n + 1 +j - i$.

Finally, we note that for $1 \le r, j \le n$, the module $M^{\overline{j}+\overline{1}-\overline{r}}_{\overline{j}}(-d\delta_{\overline{j}\neq \overline{r}})$ has length $r$ and socle $S_j$ in degree zero; we shall make extensive use of this module later on.


\subsection{Structure of $A\grstab$}

One of the desirable features of Brauer tree algebras is that the $A$-module homomorphisms $X\rightarrow Y$ can be determined combinatorially from the composition towers of $X$ and $Y$, allowing quick and easy computation of morphisms.  For a more general and explicit description of this procedure, we refer to Crawley-Boevey \cite{crawley1989maps}.  These techniques generalize easily to graded modules.

The following results about $A\grstab$ follow from straightforward computation and are well-known.  We state them without proof.

\begin{proposition}
\label{grstable indecomposables}
The (distinct) indecomposable objects of $A\grstab$ are precisely $M^{\overline{i}}_{\overline{j}} (k)$, for any $\overline{i}, \overline{j} \in \mathbb{Z}/n\mathbb{Z}, k \in \mathbb{Z}$.
\end{proposition}

\begin{proposition}
\label{grstable Hom formulas}
~\\
$dim\Hom_{A\grstab}(M^{\overline{a}}_{\overline{b}}, M^{\overline{i}}_{\overline{j}}(k)) =
\begin{cases}
1 & \text{ if } \langle a, j \rangle \subset \langle i, b \rangle \text{ and } k = -d\delta_{\overline{a}< \overline{i}}\\
0 & \text{ otherwise }
\end{cases}$
\end{proposition}

We shall refer to the statement $\langle a, j \rangle \subset \langle i , b \rangle$ as the \textbf{arc containment condition}.

For describing composition, it will be helpful to choose a collection of generators for the above Hom spaces.  Fortunately, there are natural choices.

\begin{definition}
Let $1 \le a,b, i, j \le n$, and let $l$ be the length of $M^{\overline{i}}_{\overline{j}}$.

For $\overline{i} \neq \overline{j}$, define the \textbf{canonical surjection}
\begin{eqnarray*}
\begin{tikzcd}
p^{\overline{i}}_{\overline{j}}: M^{\overline{i}}_{\overline{j}} = e_{\overline{i}}A/e_{\overline{i}}J^l  \arrow[r, two heads] & e_{\overline{i}}A/e_{\overline{i}}J^{l-1} = M^{\overline{i}}_{\overline{j}-\overline{1}}\\
e_{\overline{i}} \arrow[r, mapsto] & e_{\overline{i}}
\end{tikzcd}
\end{eqnarray*}

For $\overline{i} \neq \overline{j} + \overline{1}$, define the \textbf{canonical injection}
\begin{eqnarray*}
\begin{tikzcd}
\iota^{\overline{i}}_{\overline{j}}: M^{\overline{i}}_{\overline{j}} = e_{\overline{i}}A/ e_{\overline{i}}J^l \arrow[r, "\sim"] & (e_{\overline{i}-\overline{1}}J/ e_{\overline{i}-\overline{1}}J^{l+1})(-d\delta_{\overline{i} = \overline{1}}) \arrow[r, hook] & M^{\overline{i}-\overline{1}}_{\overline{j}}(-d\delta_{\overline{i} = \overline{1}})\\
e_{\overline{i}} \arrow[r, mapsto] & e_{\overline{i}-\overline{1}}a_{\overline{i}-\overline{1}} &
\end{tikzcd}
\end{eqnarray*}

For $<a, j> \subset <i, b>$, define the \textbf{canonical map} $\alpha^{\overline{a}, \overline{i}}_{\overline{b}, \overline{j}}: M^{\overline{a}}_{\overline{b}} \rightarrow M^{\overline{i}}_{\overline{j}}(-d\delta_{\overline{a}< \overline{i}})$ by
\begin{equation*}
\alpha^{\overline{a}, \overline{i}}_{\overline{b}, \overline{j}} = \iota^{\overline{i} + \overline{1}}_{\overline{j}}(-d\delta_{\overline{i}+\overline{1}> \overline{a}}) \cdots \iota^{\overline{a} - \overline{1}}_{\overline{j}}(-d\delta_{\overline{a} -\overline{1} > \overline{a}}) \circ \iota^{\overline{a}}_{\overline{j}} \circ p^{\overline{a}}_{\overline{j}+\overline{1}} \cdots p^{\overline{a}}_{\overline{b}}
\end{equation*}
\end{definition}

Note, in particular, that $\alpha^{\overline{a}, \overline{a}}_{\overline{b}, \overline{b}}$ is the identity map.

\begin{proposition}
\label{grstable composition formulas}
The indecomposable maps in $A\grstab$ are precisely the canonical surjections and injections.  Composition in $A\grstab$ is given by the formula:\\
$\alpha^{\overline{c}, \overline{e}}_{\overline{d}, \overline{f}}(-d\delta_{\overline{a} < \overline{c}})\circ \alpha^{\overline{a}, \overline{c}}_{\overline{b}, \overline{d}} =
\begin{cases}
\alpha^{\overline{a}, \overline{e}}_{\overline{b}, \overline{f}} & \text{ if } \langle a, f \rangle \subset \langle e, b \rangle\\
0 & \text{ otherwise }
\end{cases}$
\end{proposition}

\begin{proposition}
\label{grstable Omega formulas}
In $A\grstab$, the following formulas hold:
\begin{align}
\label{gOf1} &\Omega(M^{\overline{i}}_{\overline{j}}) = M^{\overline{j}+\overline{1}}_{\overline{i}}(d\delta_{\overline{j}+\overline{1}\le \overline{i}}) & \\
\label{gOf2} &\Omega^{-1}(M^{\overline{i}}_{\overline{j}}) = M^{\overline{j}}_{\overline{i}-\overline{1}}(-d\delta_{\overline{i}\le \overline{j}}) & \\
\label{gOf3} &\Omega^{2k}(M^{\overline{i}}_{\overline{j}}) = M^{\overline{i}+\overline{k}}_{\overline{j}+\overline{k}}(d(k + \delta_{\overline{n}+\overline{1}-\overline{k} \le \overline{i}})) & \text{ for } 1 \le k \le n.\\
\label{gOf4} &\Omega^{2k-1}(M^{\overline{i}}_{\overline{j}}) = M^{\overline{j}+\overline{k}}_{\overline{i}+\overline{k}-\overline{1}}(d(k + \delta_{\overline{n}+\overline{1}-\overline{k} \le \overline{i}} - \delta_{\overline{i}+\overline{k} \le \overline{j}+\overline{k}})) & \text{ for } 1 \le k \le n.\\
\label{gOf5} &\Omega^{-2k}(M^{\overline{i}}_{\overline{j}}) = M^{\overline{i}-\overline{k}}_{\overline{j}-\overline{k}}(-d(k + \delta_{\overline{i} \le \overline{k}})) & \text{ for } 1 \le k \le n.\\
\label{gOf6} &\Omega^{-2k+1}(M^{\overline{i}}_{\overline{j}}) = M^{\overline{j}-\overline{k}+1}_{\overline{i}-\overline{k}}(-d(k + \delta_{\overline{i}\le \overline{k}} - \delta_{\overline{j} - \overline{k} + \overline{1} \le \overline{i} - \overline{k}})) & \text{ for } 1 \le k \le n.
\end{align}
Analogous formulas hold for the $\alpha^{\overline{a}, \overline{i}}_{\overline{b}, \overline{j}}$.
\end{proposition}

\begin{proposition}
\label{grstable triangle formulas}
$\alpha^{\overline{a}, \overline{i}}_{\overline{b}, \overline{j}}: M^{\overline{a}}_{\overline{b}} \rightarrow M^{\overline{i}}_{\overline{j}}(-d\delta_{\overline{a}< \overline{i}})$ can be completed into the exact triangle:\\
\begin{eqnarray*}
\begin{tikzcd}
M^{\overline{a}}_{\overline{b}} \arrow[d, "\alpha^{\overline{a}, \overline{i}}_{\overline{b}, \overline{j}}"]\\
M^{\overline{i}}_{\overline{j}}(-d\delta_{\overline{a}< \overline{i}}) \arrow[d, "h_1"]\\
\delta_{\overline{a} \neq \overline{i}}M^{\overline{i}}_{\overline{a}-\overline{1}}(-d\delta_{\overline{a} < \overline{i}}) \oplus \delta_{\overline{b}\neq \overline{j}}M^{\overline{b}}_{\overline{j}}(-d\delta_{\overline{a} \le \overline{b}}) \arrow[d, "h_2"]\\
M^{\overline{b}}_{\overline{a}-\overline{1}}(-d\delta_{\overline{a}\le \overline{b}})
\end{tikzcd} 
\end{eqnarray*}
where $h_1 = \begin{pmatrix} \delta_{\overline{a} \neq \overline{i}}\alpha^{\overline{i}, \overline{i}}_{\overline{j}, \overline{a} - \overline{1}}(-d\delta_{\overline{a}< \overline{i}})\\ \delta_{\overline{b}\neq \overline{j}} \alpha^{\overline{i}, \overline{b}}_{\overline{j}, \overline{j}}(-d\delta_{\overline{a}< \overline{i}}) \end{pmatrix}$, and\\
$h_2 = \begin{pmatrix} \delta_{\overline{a} \neq \overline{i}}\alpha^{\overline{i}, \overline{b}}_{\overline{a} - \overline{1}, \overline{a} - \overline{1}}(-d\delta_{\overline{a}< \overline{i}}) & \delta_{\overline{b}\neq \overline{j}} \alpha^{\overline{b}, \overline{b}}_{\overline{j}, \overline{a} - \overline{1}}(-d\delta_{\overline{a} \le \overline{b}}) \end{pmatrix}$
\end{proposition}


\subsection{Structure of $A\dgstab$}
Since $\Omega \cong (-1)$ in $A\dgstab$, and $\Omega$ is periodic in $A\grstab$, it follows that $(1)$ is periodic in $A\dgstab$.  The period depends both on $n$ and $d$.  This period is the same for all indecomposable modules except when $n$ is odd, in which case the indecomposable modules of length $\frac{n+1}{2}$ have their period halved.

\begin{proposition}
\label{dgstable periodicity}
In $A\dgstab$, $M^{\overline{i}}_{\overline{j}} \cong M^{\overline{i}}_{\overline{j}}((n+1)d + 2n)$ for all $\overline{i}, \overline{j} \in \mathbb{Z}/n\mathbb{Z}$.  If $n$ is odd, then we also have $M^{\overline{i}}_{\overline{i} + \overline{\frac{n-1}{2}}} \cong M^{\overline{i}}_{\overline{i} + \overline{\frac{n-1}{2}}}(\frac{(n+1)d}{2} + n)$.
\end{proposition}
\begin{proof}
By Proposition \ref{grstable Omega formulas} we have that $M^{\overline{i}}_{\overline{j}}(-2n) \cong \Omega^{2n}M^{\overline{i}}_{\overline{j}} = M^{\overline{i}}_{\overline{j}}(d(n + 1))$, from which the first formula follows.  Similarly, if $n$ is odd, then $M^{\overline{i}}_{\overline{i} + \overline{\frac{n-1}{2}}}(-n) \cong \Omega^nM^{\overline{i}}_{\overline{i} + \overline{\frac{n-1}{2}}} = M^{\overline{i}}_{\overline{i} + \overline{\frac{n-1}{2}}}(\frac{(n+1)d}{2})$, from which the second formula follows.
\end{proof}

\begin{definition}
\label{dgstable period and shift iso}
Define the \textbf{period} of $r \in \{1, \cdots, n\}$ to be\\
\begin{equation*}
P(r) =
\begin{cases}
(n+1)d + 2n & \text{ if } r \neq \frac{n+1}{2}\\
\frac{(n+1)d}{2} + n & \text{ if }  r = \frac{n+1}{2}
\end{cases}
\end{equation*}

We also define the \textbf{period} of $M^{\overline{i}}_{\overline{j}}$ to be $P(l(M^{\overline{i}}_{\overline{j}}))$.  We define the period $P(X)$ of an arbitrary object $X$ to be the maximum period of its indecomposable components.  When we do not wish to emphasize the dependence on the length of the module, we will simply write $P = (n+1)d + 2n$.

For any $X\in A\dgstab$, let $\psi: X \rightarrow X((n+1)d + 2n)$ denote the map induced by the natural isomorphism $id \rightarrow ((n+1)d +2n)$, whose unique nonzero component is the identity map in degree $-2n$.  For any $X\in A\dgstab$ that can be expressed as a direct sum of modules of length $\frac{n+1}{2}$, let $\psi^{1/2}: X \rightarrow X(\frac{(n+1)d}{2} + n)$ denote the isomorphism whose unique nonzero component is the identity map in degree $-n$.
\end{definition}

Thus Proposition \ref{dgstable periodicity} states that for any $1 \le i, j \le n$, $M^{\overline{i}}_{\overline{j}} \cong M^{\overline{i}}_{\overline{j}}(P(r))$ in $A\dgstab$, where $r = j+1-i$ is the length of $M^{\overline{i}}_{\overline{j}}$.

One consequence of periodicity is that we can express any $M^{\overline{i}}_{\overline{j}}$ as a suitable shift of some $M^{\overline{1}}_{\overline{l}}$.  Furthermore, $l$ can always be chosen to lie in the range $1 \le l \le \frac{n+1}{2}$, since $l(\Omega M^{\overline{i}}_{\overline{j}}) = n+1-l(M^{\overline{i}}_{\overline{j}})$.

\begin{proposition}
\label{dgstable Omega formulas}
Let $1\le i, r \le n$ and $1 \le l \le \frac{n+1}{2}$.  The following identities hold in $A\dgstab$:
\begin{align}
\label{dOf1}& M^{\overline{i}}_{\overline{i}+\overline{r}-\overline{1}} \cong M^{\overline{1}}_{\overline{r}}(-(d+2)(i-1))\\
\label{dOf2}& M^{\overline{1}}_{\overline{r}} \cong M^{\overline{1}}_{\overline{n}+\overline{1}-\overline{r}}((d+2)(n+1-r) -1)\\
\label{dOf3}& M^{\overline{i}}_{\overline{i}+\overline{n}-\overline{l}} \cong M^{\overline{1}}_{\overline{l}}(-(d+2)(n+i-l) + 1)
\end{align}
\end{proposition}

\begin{proof}
We first show (\ref{dOf1}).  If $i = 1$, we are done.  Otherwise, Equation (\ref{gOf3}) of Proposition \ref{grstable Omega formulas} yields
\begin{equation*}
\begin{split}
M^{\overline{1}}_{\overline{r}}(-2(i-1)) & \cong \Omega^{2(i-1)}M^{\overline{1}}_{\overline{r}}\\
& = M^{\overline{i}}_{\overline{i}+\overline{r}-\overline{1}}(d(i-1))
\end{split}
\end{equation*}
from which (\ref{dOf1}) follows.

Applying (\ref{dOf1}) and (\ref{gOf2}), we obtain
\begin{equation*}
\begin{split}
M^{\overline{1}}_{\overline{n}+\overline{1}-\overline{r}}(1) & \cong \Omega^{-1}(M^{\overline{1}}_{\overline{n}+\overline{1}-\overline{r}})\\
 & = M^{\overline{n}+\overline{1}-\overline{r}}_{\overline{n}}(-d)\\
 & \cong M^{\overline{1}}_{\overline{r}}(-(d+2)(n-r) - d)\\
 & = M^{\overline{1}}_{\overline{r}}(-(d+2)(n+1-r) +2)
\end{split}
\end{equation*}
from which (\ref{dOf2}) follows.

(\ref{dOf3}) follows immediately from (\ref{dOf1}) and (\ref{dOf2}).
\end{proof}

We immediately obtain the following corollary:

\begin{proposition}
\label{dgstable indecomposables weak}
Every indecomposable object of $A\dgstab$ is isomorphic to one of the following:\\
1)  $M^{\overline{1}}_{\overline{l}}(k)$ for $1 \le l < \frac{n+1}{2}$ and $0 \le k < P$\\
2)  $M^{\overline{1}}_{\overline{\frac{n+1}{2}}}(k)$ for $0 \le k < \frac{P}{2}$ (if $n$ is odd)
\end{proposition}

\begin{remark}
The above list of objects are in fact pairwise non-isomorphic.  We shall prove this in Theorem \ref{dgstable morphism formula}.
\end{remark}

\begin{proof}
By applying the identities in Proposition \ref{dgstable Omega formulas}, we can express any modules $M^{\overline{i}}_{\overline{j}}$ as $M^{\overline{1}}_{\overline{l}}(k)$ for some $k \in \mathbb{Z}$ and $1 \le l \le \frac{n+1}{2}$.  By Proposition \ref{dgstable periodicity}, we can reduce $k$ mod $P(l)$ until $k$ lies in the desired range.
\end{proof}

Since $A$ is a Nakayama algebra, Lemma \ref{uniserial} guarantees that every morphism between indecomposable objects $X$ and $Y$ can be represented by a map $X \rightarrow \Omega^mY(m)$ in $A\grstab$ for some unique $m$.  To compute $\Hom_{A\dgstab}(M^{\overline{1}}_{\overline{l}}, M^{\overline{1}}_{\overline{r}}(k))$, we must determine which $M^{\overline{i}}_{\overline{j}}(m)$ admit maps from $M^{\overline{1}}_{\overline{l}}$ in the graded stable category, then express such $M^{\overline{i}}_{\overline{j}}(m)$ as $M^{\overline{1}}_{\overline{r}}(k)$ using the formulas in Proposition \ref{dgstable Omega formulas}.

\begin{proposition}
\label{tier inequality}
Let $1\le l, r, j \le n$.  Then\\
$dim\Hom_{A\grstab}(M^{\overline{1}}_{\overline{l}}, M^{{\overline{j}}+\overline{1}-\overline{r}}_{\overline{j}}(k)) =
 \begin{cases}
 1 & \text{ if } max(1, r + l - n) \le j \le min(r, l)\\
 & \phantom{A}\text{ and } k = -d \delta_{\overline{j} \neq \overline{r}}\\
 0 & \text{ otherwise}
 \end{cases}$
\end{proposition}

\begin{proof}
By Proposition \ref{grstable Hom formulas}, the dimension of the Hom space is nonzero if and only if $\langle 1, j \rangle \subset \langle j+1-r, l \rangle$ and $k = -d\delta_{\overline{1} < \overline{j}+\overline{1}-\overline{r}} = -d\delta_{\overline{j} \neq \overline{r}}$.  Thus it is enough to show that the inequalities in the statement are equivalent to the arc containment condition.  The arc containment condition holds if and only if $e^{\frac{2\pi \sqrt{-1}}{n}j} \in \langle 1, l\rangle$ and $e^{\frac{2\pi \sqrt{-1}}{n}(j+1-r)} \in \langle l+1, 1\rangle$.  This is equivalent to the chain of inequalities
\begin{equation*}
l - n + 1 \le j + 1 - r \le 1 \le j \le l
\end{equation*}
The first two inequalities are equivalent to $l + r - n \le j \le r$.  Combining these with the last two inequalities, we see the system is equivalent to $max(1, l + r - n) \le j \le min(r, l)$.
\end{proof}

We are now ready to give a complete description of the morphisms $M^{\overline{1}}_{\overline{l}} \rightarrow M^{\overline{1}}_{\overline{r}}(k)$ in $A\dgstab$, for $1 \le l, r \le \frac{n+1}{2}$.  It is helpful to organize the morphisms into two families.  The "short" morphisms are those of the form $f_{l, r, j}: M^{\overline{1}}_{\overline{l}} \rightarrow M^{\overline{j}+\overline{1}-\overline{r}}_{\overline{j}}(m)$; note that the codomain is represented by a module of length $r \le \frac{n+1}{2}$.  The "long" morphisms are of the form $g_{l, r, j}: M^{\overline{1}}_{\overline{l}} \rightarrow M^{\overline{j}+\overline{r}}_{\overline{j}}(m)$; here the codomain has length $n+1-r \ge \frac{n+1}{2}$.  

\begin{theorem}[Structure of $A\dgstab$]
\label{dgstable morphism formula}
Let $1 \le l, r \le \lfloor \frac{n+1}{2} \rfloor$ and $0 \le k < P(r)$.  Then $\Hom_{A\dgstab}(M^{\overline{1}}_{\overline{l}}, M^{\overline{1}}_{\overline{r}}(k))$ has dimension at most $1$ and is spanned by:
\begin{equation*}
f_{l, r, j}: M^{\overline{1}}_{\overline{l}} \xrightarrow{\alpha^{\overline{1}, \overline{j} + \overline{1} - \overline{r}}_{\overline{l}, \overline{j}}} M^{\overline{j}+\overline{1}-\overline{r}}_{\overline{j}}(-d \delta_{\overline{j}\neq \overline{r}}) \cong M^{\overline{1}}_{\overline{r}}(k)\\
\end{equation*}
\begin{center}
if $k \equiv (d+2)(r-j) \mod P(r)$ for some $1 \le j \le min(r, l)$
\end{center}

\begin{equation*}
g_{l,r,j}: M^{\overline{1}}_{\overline{l}} \xrightarrow{\alpha^{\overline{1}, \overline{j} + \overline{r}}_{\overline{l}, \overline{j}}} M^{\overline{j}+\overline{r}}_{\overline{j}}(-d\delta_{\overline{j}\neq \overline{n}+\overline{1}-\overline{r}}) \cong M^{\overline{1}}_{\overline{r}}(k)
\end{equation*}
\begin{center}
if $k \equiv (d+2)(n+1-j) -1 \mod P(r)$ for some $max(1, 1 + l - r) \le j \le l$
\end{center}

\begin{equation*}
0 \text{\qquad otherwise}
\end{equation*}

$f_{l, r, j}$ is an isomorphism if and only if $l = r = j$, in which case it is the identity map.  $g_{l, r, j}$ is an isomorphism if and only if $l = r = j = \frac{n+1}{2}$, in which case $g_{l, l, l} = \psi^{1/2}$.  In particular, the indecomposable modules listed in Proposition \ref{dgstable indecomposables weak} are pairwise non-isomorphic.

For $r = \frac{n+1}{2}$ and any value of $l$, the morphisms $f_{l, r, j}$ and $g_{l, r, j}$ are defined for the same values of $j$ and represented by the same morphism in $A\grstab$.  More specifically, for each such $j$,
\begin{equation}
\label{dmf5}
g_{l, r, j} = \psi^{1/2} \circ f_{l, r, j}
\end{equation}
For $l = \frac{n+1}{2}$ and any value of $r$, the morphisms $f_{l,r,j+r-\frac{n+1}{2}}$ and $g_{l,r,j}$ are defined for the same values of $j$, and their unique nonzero components differ only by an application of $\Omega^n$ and a grading shift.  More precisely, for each such $j$,
\begin{equation}
\label{dmf6}
\psi \circ f_{l,r,j+r-\frac{n+1}{2}} = g_{l,r,j} (\frac{(n+1)d}{2} + n)\circ \psi^{1/2}
\end{equation}
Apart from the above identities, all the $f_{l,r,j}$ and $g_{l,r,j}$ are distinct, in the sense that their unique nonzero components cannot be transformed into one another by applying powers of $\Omega$ and grading shifts.

The indecomposable morphisms in $A\dgstab$ are, up to shifts, $f_{l, l-1, l-1}$ for $1 < l \le \lfloor \frac{n+1}{2} \rfloor$, $f_{l, l+1, l}$ for $1 \le l < \lfloor \frac{n+1}{2} \rfloor$, $g_{\frac{n}{2},\frac{n}{2}, \frac{n}{2}}$ for $n$ even, and $g_{\frac{n+1}{2}, \frac{n+1}{2}-1, \frac{n+1}{2}}$ for $n$ odd.

Composition of morphisms in $A\dgstab$ is determined by composing the unique nonzero morphisms in $A\grstab$.  In particular, we have the formulas:
\begin{equation}
\label{dmf1}
f_{r,c,q}((d+2)(r-j))\circ f_{l, r, j} =
\begin{cases}
f_{l, c, q+j-r} & \text{ if } 1 \le q+j-r \le\\
& \quad min(c,l)\\
0 & \text{ otherwise}
\end{cases}
\end{equation}

\begin{equation}
\label{dmf2}
g_{r,c,q}((d+2)(r-j))\circ f_{l, r, j} =
\begin{cases}
g_{l, c, q+j-r} & \text{ if } max(1, 1+l-c) \le\\
& \quad q+j-r \le l\\
0 & \text{ otherwise}
\end{cases}
\end{equation}

\begin{equation}
\label{dmf3}
f_{r,c,q}((d+2)(n+1-j)-1) \circ g_{l,r,j} =
\begin{cases}
g_{l, c, q+j-c} & \text{ if } max(1, 1+l-c) \le\\
& \quad q+j-c \le l\\
0 & \text{ otherwise}
\end{cases}
\end{equation}

\begin{equation}
\label{dmf4}
g_{r,c,q}((d+2)(n+1-j)-1) \circ g_{l,r,j} =
\begin{cases}
\psi \circ f_{l, c, q+j+c-(n+1)} & \text{ if } l < q+j  \le\\
& \quad n+1 \le\\
& \quad q+j+c-1\\
0 & \text{ otherwise}
\end{cases}
\end{equation}
\end{theorem}

\begin{proof}
By Lemma \ref{uniserial}, $\Hom_{A\grstab}(M^{\overline{1}}_{\overline{l}}, \Omega^mM^{\overline{1}}_{\overline{r}}(k+m))$ is nonzero for at most one $m$.  Thus by Proposition \ref{tier inequality}, $\Hom_{A\dgstab}(M^{\overline{1}}_{\overline{l}}, M^{\overline{1}}_{\overline{r}}(k))$ is nonzero if and only if $\Omega^mM^{\overline{1}}_{\overline{r}}(k+m)) \cong M^{\overline{j}+\overline{1}-\overline{x}}_{\overline{j}}(-d \delta_{{\overline{j}}\neq \overline{x}})$ in $A\grstab$, for some $1 \le x \le n$, $max(1, x + l - n) \le j \le min(x, l)$, and $m\in \mathbb{Z}$.  The only possibles values for the length of $\Omega^mM^{\overline{1}}_{\overline{r}}(k+m))$ are $r$ and $n+1-r$, so we need only consider the cases $x = r$ and $x = n+1-r$, and verify that $M^{\overline{j}+\overline{1}-\overline{x}}_{\overline{j}}(-d \delta_{\overline{j}\neq \overline{x}}) \cong M^{\overline{1}}_{\overline{r}}(k)$ for the desired value of $k$.

If $x = r$, then the condition on $j$ simplifies to $1 \le j \le min(r, l)$.  For each such $j$ we obtain the morphism $f_{l, r, j}$ whose nonzero component is $\alpha^{\overline{1},\overline{j}+\overline{1}-\overline{r}}_{\overline{l}, \overline{j}}: M^{\overline{1}}_{\overline{l}} \rightarrow M^{\overline{j}+\overline{1}-\overline{r}}_{\overline{j}}(-d \delta_{\overline{j}\neq \overline{r}})$.  If $j \neq r$, by substituting $i = n+1 - (r-j)$ into Equation (\ref{dOf1}) of Proposition \ref{dgstable Omega formulas} we obtain
\begin{equation*}
\begin{split}
M^{\overline{j}+\overline{1}-\overline{r}}_{\overline{j}}(-d) & \cong M^{\overline{1}}_{\overline{r}}(-(d+2)(n-(r-j))-d)\\
& = M^{\overline{1}}_{\overline{r}}((d+2)(r-j) -(n+1)d -2n)\\
& \cong M^{\overline{1}}_{\overline{r}}((d+2)(r-j))
\end{split}
\end{equation*}
If $j = r$, then the desired identity is immediate.

If $x = n+1-r$, then the condition on $j$ simplifies to $max(1, 1 + l - r) \le j \le l$.  For each such $j$ we obtain the morphism $g_{l, r, j}$ whose nonzero component is $\alpha^{\overline{1}, \overline{j}+\overline{r}}_{\overline{l}, \overline{j}}: M^{\overline{1}}_{\overline{l}} \rightarrow M^{\overline{j}+\overline{r}}_{\overline{j}}(-d \delta_{\overline{j}\neq \overline{n}+\overline{1}-\overline{r}})$.  If $j\neq n+1-r$, applying Equation (\ref{dOf3}) of Proposition \ref{dgstable Omega formulas} with the substitutions $i \mapsto j+r$ and $l\mapsto r$, we obtain
\begin{equation*}
\begin{split}
M^{\overline{j}+\overline{r}}_{\overline{j}}(-d) & \cong M^{\overline{1}}_{\overline{r}}(-(d+2)(n+j) + 1-d)\\
& \cong M^{\overline{1}}_{\overline{r}}(-(d+2)(n+j) + 1-d + 2(n+1)d + 4n)\\
& = M^{\overline{1}}_{\overline{r}}((d+2)(n+1-j)-1)
\end{split}
\end{equation*}
If $j = n+1-r$, the restrictions on $j$, $r$, and $l$ imply that $j = l = r = \frac{n+1}{2}$; the desired identity then follows from direct substitution and the fact that $P(r) = \frac{(n+1)d}{2} + n$.  This establishes the descriptions of the Hom spaces.

We now determine the isomorphisms of $A\dgstab$.  Note that a morphism is an isomorphism in $\mathcal{C}(A)$ if and only if its unique nonzero component is an isomorphism in $A\grstab$.  If $f_{l,r,j}$ is an isomorphism, $\alpha^{\overline{1},\overline{j}+\overline{1}-\overline{r}}_{\overline{l}, \overline{j}}: M^{\overline{1}}_{\overline{l}} \rightarrow M^{\overline{j}+\overline{1}-\overline{r}}_{\overline{j}}(-d \delta_{\overline{j}\neq \overline{r}})$ must also be an isomorphism.  From Proposition \ref{grstable indecomposables}, we have that $j=r=l$.  Conversely, direct substitution shows that $f_{l,l,l}$ is the identity map.

Similarly, if $g_{l,r,j}$ is an isomorphism, then its nonzero component $\alpha^{\overline{1}, \overline{j} + \overline{r}}_{\overline{l}, \overline{j}}: M^{\overline{1}}_{\overline{l}} \rightarrow M^{\overline{j}+\overline{r}}_{\overline{j}}(-d\delta_{\overline{j}\neq \overline{n}+\overline{1}-\overline{r}})$ is also an isomorphism.  Proposition \ref{grstable indecomposables} then forces $j = r = l = \frac{n+1}{2}$.  Conversely, if $j= r = l = \frac{n+1}{2}$, direct substitution shows that the nonzero component of $g_{l,l,l}: M^{\overline{1}}_{\overline{l}} \rightarrow M^{\overline{1}}_{\overline{l}}(\frac{(n+1)d}{2}+n)$ is the identity map.  Thus $g_{l, l, l} = \psi^{1/2}$.

This completes the description of the isomorphisms of $A\dgstab$.  It follows immediately that $M^{\overline{1}}_{\overline{l}} \cong M^{\overline{1}}_{\overline{r}}(k)$ if and only if $l = r$ and $k \equiv 0 \mod P(l)$.  Thus the indecomposable objects of Proposition \ref{dgstable indecomposables weak} are pairwise non-isomorphic.

If $r = \frac{n+1}{2}$, the identity $g_{l, r, j} = \psi^{1/2} \circ f_{l, r, j}$ follows immediately from direct substitution, as does the fact that $f$ and $g$ are defined for the same values of $j$.  (We note that both morphisms have the same domain, codomain, and nonzero component.)  If $l = \frac{n+1}{2}$, the corresponding statement follows from similar computations.

To show distinctness, suppose that we can transform the nonzero component $f_{l,r,j}$ into that of $g_{l',r',j'}$ by applying grading shifts and $\Omega^m$ for some even integer $m$.  We shall show that $l = l'$, $j=j'$, and $r = r' = \frac{n+1}{2}$.  In order for the domains to be equal, we must have that $l = l'$ and $m$ is a multiple of $2n$.  By using periodicity of $\Omega$ and changing the grading shift, we can also assume without loss of generality that $m=0$.  For the codomains to be equal, we must have in particular that $M^{\overline{j}+\overline{1}-\overline{r}}_{\overline{j}} \cong M^{\overline{j'}+\overline{r}'}_{\overline{j'}}$ in $A\grstab$, hence $j=j'$ and $r = r' = \frac{n+1}{2}$.

Similarly, suppose we can transform the nonzero component of $f_{l,r,j}$ into that of $f_{l',r',j'}$ by applying a grading shift and $\Omega^m$ for some even $m$.  By observing the domain and codomain we once again see that $m$ can be taken to be zero, and we immediately obtain $l = l'$, $r=r'$, and $j=j'$.  The same argument also applies to $g_{l,r,j}$ and $g_{l', r', j'}$.

Next, suppose that we can transform the nonzero component $f_{l,r,j}$ into that of $g_{l',r',j'}$ by applying grading shifts and $\Omega^m$ for some odd integer $m$.  We shall show that $l=l'=\frac{n+1}{2}$, $r=r'$, and $j = j'+r' -\frac{n+1}{2}$.  In order for the domains to be equal, their lengths, $n+1-l$ and $l'$, respectively, must be equal.  This implies that $n$ is odd and $l = l' = \frac{n+1}{2}$.  Furthermore, $m$ must be an odd multiple of $n$; without loss of generality, we may assume that $m=n$ by periodicity.  Observing the codomains, we must have that $M^{\overline{j}+\overline{\frac{n+1}{2}}}_{\overline{j}-\overline{r}+\overline{\frac{n+1}{2}}} \cong M^{\overline{j'}+\overline{r'}}_{\overline{j'}}$ in $A\grstab$.  Comparing the bottom indices, we have that $\overline{j} = \overline{j'}+\overline{r}-\overline{\frac{n+1}{2}}$.  Comparing the top indices and using the previous equation, we have that $\overline{r} = \overline{r'}$, hence $r = r'$.  From the restrictions on the ranges of the indices, we deduce that $j = j'+r'-\frac{n+1}{2}$, as desired.

If we can transform the nonzero component of $f_{l,r,j}$ into that of $f_{l',r',j'}$ by applying a grading shift and $\Omega^m$ for some odd $m$, by considering the lengths of the domain and codomain, we must have that $l = l' = r = r' = \frac{n+1}{2}$.  We can again assume that $m=n$.  Observing the codomains, we must have that $M^{\overline{j}+\overline{\frac{n+1}{2}}}_{\overline{j}-\overline{r}+\overline{\frac{n+1}{2}}} \cong M^{\overline{j'}+\overline{1}-\overline{r'}}_{\overline{j'}}$.  It follows that $j = j'$.  A parallel argument applies to $g_{l,r,j}$ and $g_{l',r',j'}$.  Thus all the morphisms are distinct, except for the listed identities.

A morphism $M^{\overline{1}}_{\overline{l}} \rightarrow M^{\overline{1}}_{\overline{r}}(k)$ is indecomposable if and only if its nonzero component is indecomposable in $A\grstab$.  Thus, up to a degree shift, the indecomposable morphisms of $A\dgstab$ are those $f_{l,r,j}$ and $g_{l,r,j}$ whose nonzero component is a canonical injection or surjection.  Since the domain must be $M^{\overline{1}}_{\overline{l}}$, the only possible values for these nonzero components are $p^{\overline{1}}_{\overline{l}}$ for $1 < l \le \frac{n+1}{2}$ and $\iota^{\overline{1}}_{\overline{l}}$ for $1 \le l \le \frac{n+1}{2}$.  From the definitions, $p^{\overline{1}}_{\overline{l}}$ is the nonzero component of $f_{l, l-1, l-1}$ for all $1 < l \le \frac{n+1}{2}$ and $\iota^{\overline{1}}_{\overline{l}}$ is the nonzero component of $f_{l, l+1, l}$ for all $1 \le l < \lfloor \frac{n+1}{2} \rfloor$.  When $n$ is even, $\lfloor \frac{n+1}{2} \rfloor = \frac{n}{2}$, and $\iota^{\overline{1}}_{\overline{\frac{n}{2}}}$ is the nonzero component of $g_{\frac{n}{2}, \frac{n}{2}, \frac{n}{2}}$.  When $n$ is odd, $\lfloor \frac{n+1}{2} \rfloor = \frac{n+1}{2}$, and $\iota^{\overline{1}}_{\overline{\frac{n+1}{2}}}$ is the nonzero component of $g_{\frac{n+1}{2}, \frac{n+1}{2}-1, \frac{n+1}{2}}$.  These are precisely the indecomposable morphisms listed in the statement.

To verify the composition formulas, we translate them into statements about $A\grstab$ and use Proposition \ref{grstable composition formulas}.

We start with Equation (\ref{dmf1}).  A tedious but straightforward computation using Proposition \ref{grstable Omega formulas} shows that the only possible nonzero component of
\begin{equation*}
f_{r,c,q}((d+2)(r-j))\circ f_{l, r, j}: M^{\overline{1}}_{\overline{l}} \rightarrow M^{\overline{1}}_{\overline{c}}((d+2)(c-(q+j-r)))
\end{equation*}
 is $\alpha^{\overline{j} + \overline{1} - \overline{r}, (\overline{q} + \overline{j} - \overline{r}) + \overline{1} - \overline{c}}_{\overline{j}, \overline{q} + \overline{j} - \overline{r}}(-d\delta_{j\neq r})\circ \alpha^{\overline{1},\overline{j} + \overline{1} - \overline{r}}_{\overline{l}, \overline{j}}$.  Then by Proposition \ref{grstable composition formulas}, this composition is nonzero if and only if $\langle 1, q+j-r\rangle \subset \langle (q+j-r) + 1-c , l \rangle$, in which case it is equal to $\alpha^{\overline{1}, (\overline{q} + \overline{j} - \overline{r}) + \overline{1} - \overline{c}}_{\overline{l}, \overline{q} + \overline{j} - \overline{r}}$.  Since the codomain of this morphism has length $c$, it follows that the resulting morphism, if nonzero, is equal to $f_{l, c, q+j-r}$.  It remains to verify that the arc containment condition is equivalent to the desired inequality.  If $q+j-r < 1$, then the restrictions on $q, j, r$, and $l$ imply that $l-n < q+j-r < 1$, hence both the desired inequality and the arc containment condition are false.  If $q+j-r \ge 1$, the restrictions on $q, j, l, r$, and $r$ imply that $1 \le q+j-r \le n$ and $c+l-n \le 1$.  We can then apply Proposition \ref{tier inequality} and conclude the arc containment condition is equivalent to the inequality $1 \le q+j-r \le min(c,l)$.  Thus Equation (\ref{dmf1}) holds.

Proceeding to Equation (\ref{dmf2}), the only possible nonzero component of
\begin{equation*}
g_{r,c,q}((d+2)(r-j))\circ f_{l, r, j}: M^{\overline{1}}_{\overline{l}} \rightarrow M^{\overline{1}}_{\overline{c}}((d+2)(n+1-(q+j-r))-1)
\end{equation*}
 is $\alpha^{\overline{j}+\overline{1}-\overline{r}, (\overline{q}+\overline{j}-\overline{r})+\overline{c}}_{\overline{j}, \overline{q}+\overline{j}-\overline{r}}(-d\delta_{j\neq r}) \circ \alpha^{\overline{1}, \overline{j} + \overline{1} -\overline{r}}_{\overline{l}, \overline{j}}$.  This composition is nonzero if and only if $\langle 1, q+j-r \rangle \subset \langle (q+j-r)+c, l \rangle$, in which case it is equal to $\alpha^{\overline{1}, (\overline{q}+\overline{j}-\overline{r})+\overline{c}}_{\overline{l}, \overline{q}+\overline{j}-\overline{r}}$.  The codomain of this component has length $n+1-c$, hence the composition, if nonzero, is equal to $g_{l, c, q+j-r}$.  The same argument as above shows that the arc containment condition is equivalent to the inequality $max(1, 1+l-c)\le q+j-r \le l$.  Thus Equation (\ref{dmf2}) holds.

For Equation (\ref{dmf3}), the only possible nonzero component of
\begin{equation*}
f_{r,c,q}((d+2)(n+1-j)-1) \circ g_{l,r,j}: M^{\overline{1}}_{\overline{l}} \rightarrow M^{\overline{1}}_{\overline{c}}((d+2)(n+1-(q+j-c))-1)
\end{equation*}
 is $\alpha^{\overline{j}+\overline{r}, \overline{q}+\overline{j}}_{\overline{j}, \overline{q}+\overline{j}-\overline{c}}(-d \delta_{\overline{j}\neq \overline{n}+\overline{1}-\overline{r}}) \circ \alpha^{\overline{1}, \overline{j}+\overline{r}}_{\overline{l}, \overline{j}}$. This composition is nonzero if and only if $\langle 1, q+j-c \rangle \subset \langle q+j, l \rangle$, in which case it is equal to $\alpha^{\overline{1}, \overline{q} + \overline{j}}_{\overline{l}, \overline{q} + \overline{j} - \overline{c}}$.  The codomain of this component has length $n+1-c$, hence the composition, if nonzero, is equal to $g_{l,c,q+j-c}$.   The same argument as above shows that the arc containment condition is equivalent to the inequality $max(1, 1+l-c) \le q+j-c \le l$.  Thus Equation (\ref{dmf3}) holds.

For Equation (\ref{dmf4}), the only possible nonzero component of
\begin{equation*}
g_{r,c,q}((d+2)(n+1-j)-1) \circ g_{l,r,j}: M^{\overline{1}}_{\overline{l}} \rightarrow M^{\overline{1}}_{\overline{c}}((d+2)(2(n+1)-q-j)-2)
\end{equation*}
 is $\alpha^{\overline{j}+\overline{r},\overline{q}+\overline{j}}_{\overline{j}, \overline{q}+\overline{j}+\overline{c}-\overline{1}}(-d\delta_{\overline{j}\neq \overline{n}+\overline{1}-\overline{r}}) \circ \alpha^{\overline{1}, \overline{j}+\overline{r}}_{\overline{l}, \overline{j}}$.  This composition is nonzero if and only if $\langle 1, q+j+c-1 \rangle \subset \langle q+j, l \rangle$, in which case it is equal to $\alpha^{\overline{1}, \overline{q}+\overline{j}}_{\overline{l}, \overline{q}+\overline{j}+\overline{c}-\overline{1}}$.  It is clear that the desired inequality implies the arc containment condition; we now show the converse.  Due to the restrictions on $q, j$, and $c$, we have that $2 \le q+j \le q+j+c-1 \le l+n$ and $q+j \le n+1$.  Thus if $q+j \le l$, we have that $1 < q+j \le l$, and the arc containment condition fails.  Thus we must have that $l < q+j \le n+1$.  If $l < q+j+c-1 \le n+1$, then the arc containment condition fails, hence we must also have $n+1 \le q+j+c-1$.  The desired inequality follows immediately.  Thus the arc containment condition and the desired equality are equivalent.  If both hold, the codomain of the nonzero component has length $c$.  We also have that $1 \le q+j+c-(n+1) \le min(c, l)$.  Thus the composition is equal to $\psi \circ f_{l, c, q+j+c-(n+1)}$.  To explain the presence of $\psi$ in this formula, note that the grading shift of the codomain of the composition is 
\begin{equation*}
(d+2)(2(n+1)-q-j)-2 = [(d+2)(c-(q+j+c-(n+1)))] + [(n+1)d+2n]
\end{equation*}
The factor of $\psi$ accounts for the second bracketed term.
\end{proof}

Computing the cones of the morphisms in Theorem \ref{dgstable morphism formula} is straightforward, since the computations can be done in $A\grstab$.

\begin{proposition}
For all values of $l$, $r$, and $j$ for which it is defined, $f_{l, r, j}$ can be completed to the exact triangle:
\begin{eqnarray}
\label{dcf1}
\begin{tikzcd}
M^{\overline{1}}_{\overline{l}} \arrow[d, "f_{l,r,j}"]\\
M^{\overline{1}}_{\overline{r}}((d+2)(r-j)) \arrow[d, "h_1"]\\
\delta_{\overline{j}\neq \overline{r}}M^{\overline{1}}_{\overline{r}-\overline{j}}((d+2)(r-j)) \oplus \delta_{\overline{j}\neq \overline{l}}M^{\overline{1}}_{\overline{l}-\overline{j}}((d+2)(n+1-j)-1) \arrow[d, "h_2"]\\
M^{\overline{1}}_{\overline{l}}(1)
\end{tikzcd}
\end{eqnarray}
where
$h_1 = \begin{pmatrix} \delta_{\overline{j}\neq \overline{r}}f_{r, r-j, r-j}((d+2)(r-j))\\ \delta_{\overline{j}\neq \overline{l}}g_{r, l-j, r}((d+2)(r-j)) \end{pmatrix}$ and \\
$h_2 = ( \psi^{-1} \delta_{\overline{j}\neq \overline{r}}g_{r-j, l, r-j}((d+2)(r-j)),\\
\psi^{-1} \delta_{\overline{j}\neq \overline{l}}f_{l-j, l, l-j}((d+2)(n+1-j)-1) )$

For all values of $l$, $r$, and $j$ for which it is defined and such that $j+r \ge \frac{n+1}{2}$, $g_{l,r,j}$ can be completed to the exact triangle:

\begin{equation}
\label{dcf2}
\begin{tikzcd}
M^{\overline{1}}_{\overline{l}} \arrow[d, "g_{l,r,j}"]\\
M^{\overline{1}}_{\overline{r}}((d+2)(n+1-j)-1) \arrow[d, "h_3"]\\
\delta_{\overline{j} + \overline{r} \neq \overline{1}}M^{\overline{1}}_{\overline{n} + \overline{1} - (\overline{j}+\overline{r})}((d+2)(2(n+1)-(j+r))-2)
\end{tikzcd}
\end{equation}
\vspace{-20pt}
\begin{center}
$\oplus$
\end{center}
\vspace{-5pt}
\begin{equation*}
\begin{tikzcd}
\delta_{\overline{j}\neq \overline{l}}M^{\overline{1}}_{\overline{l}-\overline{j}}((d+2)(n+1-j)-1) \arrow[d, "h_4"]\\
M^{\overline{1}}_{\overline{l}}(1)
\end{tikzcd}
\end{equation*}
where
$h_3 = \begin{pmatrix} \delta_{\overline{j} + \overline{r} \neq \overline{1}}g_{r, n+1-(j+r), r}((d+2)(n+1-j)-1)\\ \delta_{\overline{j}\neq \overline{l}}f_{r, l-j, l-j}((d+2)(n+1-j)-1) \end{pmatrix}$ and \\
$h_4 = (\psi^{-2} \delta_{\overline{j} + \overline{r} \neq \overline{1}}g_{n+1-(j+r), l, n+1-(j+r)}((d+2)(2(n+1) -(j+r)) -2),\\
\psi^{-1} \delta_{\overline{j}\neq \overline{l}}f_{l-j, l, l-j}((d+2)(n+1-j)-1) )$ 

For all values of $l$, $r$, and $j$ for which it is defined and such that $j+r \le \frac{n+1}{2}$, $g_{l,r,j}$ can be completed to the exact triangle:

\begin{eqnarray}
\label{dcf3}
\begin{tikzcd}
M^{\overline{1}}_{\overline{l}} \arrow[d, "g_{l,r,j}"]\\
M^{\overline{1}}_{\overline{r}}((d+2)(n+1-j)-1) \arrow[d, "h_5"]\\
\delta_{\overline{j} + \overline{r} \neq \overline{1}}M^{\overline{1}}_{\overline{j}+\overline{r}}((d+2)(n+1)-1) \oplus \delta_{\overline{j}\neq \overline{l}}M^{\overline{1}}_{\overline{l}-\overline{j}}((d+2)(n+1-j)-1) \arrow[d, "h_6"]\\
M^{\overline{1}}_{\overline{l}}(1)
\end{tikzcd}
\end{eqnarray}
where
$h_5 = \begin{pmatrix} \delta_{\overline{j} + \overline{r} \neq \overline{1}}f_{r, j+r, r}((d+2)(n+1-j)-1)\\ \delta_{\overline{j}\neq \overline{l}}f_{r, l-j, l-j}((d+2)(n+1-j)-1) \end{pmatrix}$ and \\
$h_6 = ( \psi^{-1} \delta_{\overline{j} + \overline{r} \neq \overline{1}}f_{j+r, l, l}((d+2)(n+1)-1),\\
\psi^{-1} \delta_{\overline{j}\neq \overline{l}}f_{l-j, l, l-j}((d+2)(n+1-j)-1) )$

Triangles (\ref{dcf2}) and (\ref{dcf3}) are equivalent when $j+r = \frac{n+1}{2}$.
\end{proposition}

\begin{proof}
Note that triangles in $A\grstab$ induce triangles in $A\dgstab$.  By Proposition \ref{grstable triangle formulas}, the nonzero component of $f_{l,r,j}$ fits into an exact triangle
\begin{eqnarray*}
\begin{tikzcd}
M^{\overline{1}}_{\overline{l}} \arrow[d, "\alpha^{\overline{1}, \overline{j}+\overline{1}-\overline{r}}_{\overline{l}, \overline{j}}"]\\
M^{\overline{j} + \overline{1} - \overline{r}}_{\overline{j}}(-d\delta_{\overline{j} \neq \overline{r}}) \arrow[d, "h'_1"]\\
\delta_{\overline{j}\neq \overline{r}}M^{\overline{j} + \overline{1} -\overline{r}}_{\overline{n}}(-d) \oplus \delta_{\overline{j}\neq \overline{l}}M^{\overline{l}}_{\overline{j}}(-d) \arrow[d, "h'_2"]\\
M^{\overline{l}}_{\overline{n}}(-d)
\end{tikzcd}
\end{eqnarray*}
where $h'_1 = \begin{pmatrix} \delta_{\overline{j} \neq \overline{r}}\alpha^{\overline{j}+\overline{1}-\overline{r}, \overline{j}+\overline{1}-\overline{r}}_{\overline{j}, \overline{n}}(-d\delta_{\overline{j} \neq \overline{r}})\\
\delta_{\overline{j}\neq \overline{l}} \alpha^{\overline{j} + \overline{1}-\overline{r}, \overline{l}}_{\overline{j}, \overline{j}}(-d\delta_{\overline{j} \neq \overline{r}}) \end{pmatrix}$, and\\
$h'_2 = \begin{pmatrix} \delta_{\overline{j} \neq \overline{r}}\alpha^{\overline{j}+\overline{1}-\overline{r}, \overline{l}}_{\overline{n}, \overline{n}}(-d) & \delta_{\overline{j}\neq \overline{l}} \alpha^{\overline{l}, \overline{l}}_{\overline{j}, \overline{n}}(-d) \end{pmatrix}$.

Using the identities in Proposition \ref{grstable Omega formulas}, it follows immediately that this triangle is isomorphic to the triangle (\ref{dcf1}).

The other two cases are proved analogously.
\end{proof}

We have described the indecomposable objects and morphisms of $A\dgstab$ in Theorem \ref{dgstable morphism formula}; by two easy counting arguments ($n$ even and $n$ odd), there are $\frac{nP}{2}$ indecomposable objects in $A\dgstab$.  A description of the Auslander-Reiten quiver $A\dgstab$ follows easily.

\begin{definition}
For positive integers $N, M$, let $Q_{N, M}$ denote the quiver with vertex set $V = \mathbb{Z}/N\mathbb{Z} \times \{1, \cdots, M\}$ and arrows of the form $(x, y) \rightarrow (x+1, y+1)$ for $1 \le y < M$ and $(x, y) \rightarrow (x, y-1)$ for $1 < y \le M$.
\end{definition}

\begin{corollary}
\label{dgstable AR quiver}
If $d$ is even, the Auslander-Reiten quiver of $A\dgstab$ is the cylinder $Q_{\frac{P}{2}, n}$.  If $d$ is odd, the Auslander-Reiten quiver is the M{\"o}bius strip $Q_{P, n}/\tau$, where $\tau$ is the involution sending $(x, y)$ to $(x -y + \frac{P+n+1}{2}, n+1-y)$.  The indecomposable morphisms representing the arrows can be chosen such that all squares commute (up to powers of $\psi$), and a composition of arrows out of vertex $(x, y)$ is zero if and only if the composition contains at least $y$ vertical arrows or at least $n+1-y$ diagonal arrows.
\end{corollary}

\begin{example}
\begin{eqnarray*}
\begin{tikzcd}
3 \arrow[d]  & \bullet \arrow[d] & \bullet \arrow[d] & \bullet \arrow[d] & \bullet \arrow[d]  & \bullet \arrow[d]  & \bullet \arrow[d] & 3 \arrow[d]
\\
2 \arrow[d] \arrow[ur] & \bullet \arrow[d] \arrow[ur] & \bullet \arrow[d] \arrow[ur] & \bullet \arrow[d] \arrow[ur] & \bullet \arrow[d] \arrow[ur] & \bullet \arrow[d] \arrow[ur] & \bullet \arrow[d] \arrow[ur] & 2 \arrow[d]
\\
1 \arrow[ur] & \bullet \arrow[ur] & \bullet \arrow[ur] & \bullet \arrow[ur] & \bullet \arrow[ur] & \bullet \arrow[ur] & \bullet \arrow[ur] & 1
\end{tikzcd}
\end{eqnarray*}
\begin{center}
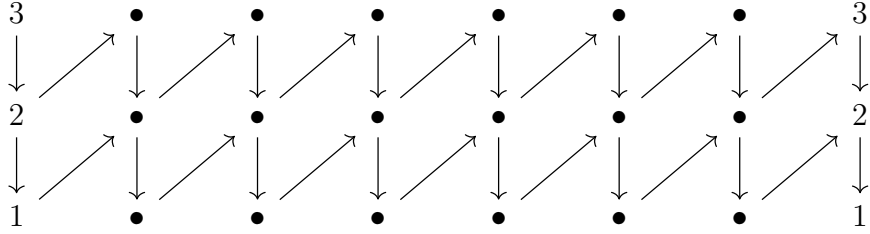
The Auslander-Reiten quiver of $A\dgstab$ for $n=3, d = 2$.\\
Corresponding numbered vertices are identified.
\end{center}

\begin{eqnarray*}
\begin{tikzcd}
& & 1 \arrow[d] & \bullet \arrow[d] & \bullet \arrow[d] & \bullet \arrow[d] & 3 \arrow[d]
\\
& 2 \arrow[d] \arrow[ur] & \bullet \arrow[d] \arrow[ur] & \bullet \arrow[d] \arrow[ur] & \bullet \arrow[d] \arrow[ur] & \bullet \arrow[d] \arrow[ur] & 2 \arrow[d]
\\
3 \arrow[ur] & \bullet \arrow[ur] & \bullet \arrow[ur] & \bullet \arrow[ur] & \bullet \arrow[ur] & \bullet \arrow[ur] & 1
\end{tikzcd}
\end{eqnarray*}
\begin{center}
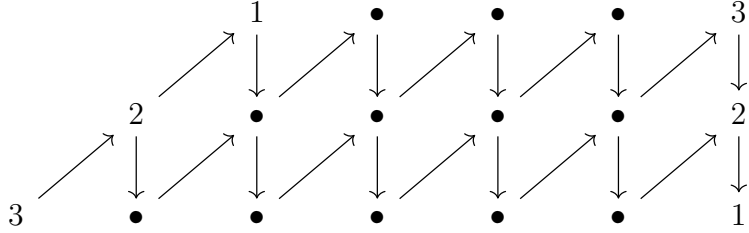
The Auslander-Reiten quiver of $A\dgstab$ for $n=3, d = 1$.\\
Corresponding numbered vertices are identified.
\end{center}
\end{example}

\begin{proof}
An easy counting argument shows that there are exactly $\frac{Pn}{2}$ indecomposable objects in $A\dgstab$ (regardless of parity of $n$), and that there are $\frac{Pn}{2}$ vertices in the corresponding candidate quivers.  We first make explicit the bijection between vertices and indecomposable objects.

For $d$ even, $x \in \mathbb{Z}/\frac{P}{2}\mathbb{Z}$, and $1 \le y \le n$, associate the vertex $(x, y)$ with the object $M^{\overline{1}}_{\overline{y}}((d+2)x)$; the isomorphism class of this object is independent of the choice of representative of $x$ since $d+2$ is even.  For $\lfloor \frac{n+1}{2} \rfloor < y \le n$, note that $M^{\overline{1}}_{\overline{y}}((d+2)x) \cong M^{\overline{1}}_{\overline{ n+1 - y }}((d+2)(x+n+1-y) - 1)$, with $1 \le n+1-y < \lfloor \frac{n+1}{2} \rfloor$.  Since $(n+1)(d+2) \equiv 2$ in $\mathbb{Z}/P\mathbb{Z}$ and both $d$ and $P$ are even, we have that $(d+2) = (2)$ is an index $2$ subgroup of $\mathbb{Z}/P\mathbb{Z}$.  Thus, for fixed $1 < y < \frac{n+1}{2}$, the above mapping sends $\{ (x, y) \mid 0 \le x < \frac{P}{2} \}$ onto $\{ M^{\overline{1}}_{\overline{y}}(2i) \mid 0 \le i < \frac{P}{2} \}$.  For fixed $1 < n+1-y < \frac{n+1}{2}$, the mapping sends $\{ (x,y) \mid 0 \le x < \frac{P}{2} \}$ onto $\{ M^{\overline{1}}_{\overline{y}}(2i+1) \mid 0 \le i < \frac{P}{2} \}$.  If $n$ is odd, then $d+2$ generates $\mathbb{Z}/\frac{P}{2} \mathbb{Z}$ and so $\{(x, \frac{n+1}{2}) \mid 0 \le x < \frac{P}{2} \}$ maps onto $\{ M^{\overline{1}}_{\overline{\frac{n+1}{2}}}(i) \mid 0 \le i < \frac{P}{2} \}$.  Thus the above mapping establishes a bijection between the indecomposable objects of $A\dgstab$ and the vertices of $Q_{\frac{P}{2}, n}$.  

When $d$ is odd, associate the vertex $(x,y)$ of $Q_{P, n}$ to the object $M^{\overline{1}}_{\overline{y}}((d+2)x)$.  Note that $d+2$ is odd and $\langle d+2 \rangle \supset \langle 2 \rangle$ in $\mathbb{Z}/P\mathbb{Z}$; therefore $d+2$ generates $\mathbb{Z}/P\mathbb{Z}$.  Thus, for each $1 \le y \le n$, the above mapping establishes a bijection between $\{(x, y) \mid 0 \le x < P\}$ and $\{M^{\overline{1}}_{\overline{y}}(i) \mid 0 \le i < P \}$.  In particular, we have a two-to-one map from the vertices of $Q_{P,n}$ to the indecomposable objects of $A\dgstab$.  Furthermore, a straightforward computation using Equation (\ref{dOf2}) of Lemma \ref{dgstable Omega formulas} shows that $(x,y)$ and $\tau(x,y)$ have the same image.  Thus the map defines a bijection between the vertices of $Q_{P,n}/\tau$ are the indecomposable objects of $A\dgstab$.

It remains to show that the edges of the candidate quivers correspond to the indecomposable morphisms.  For $d$ of arbitrary parity, every length one indecomposable object of $A\dgstab$ has one indecomposable morphism in and out; all other objects have two indecomposable morphisms in and out.  Thus the degrees of the vertices in the Auslander-Reiten quiver agree with the degrees of the vertices of the candidate quivers.  Note that both candidate quivers are symmetric about the central line $y=\frac{n+1}{2}$.  For a vertex $(x,y)$, let $\widetilde{y} = min\{y, n+1-y\}$.  For vertices $(x,y)$ which do not lie on the central line, it is straightforward to verify that shifts of the indecomposable morphism $f_{\widetilde{y}, \widetilde{y}-1, \widetilde{y}-1}$ correspond to the arrows exiting $(x,y)$ and pointing away from the central line.  Similarly, shifts of the morphism $f_{\widetilde{y}, \widetilde{y}+1, \widetilde{y}}$ correspond to the arrows pointing towards the central line, and shifts of $g_{\frac{n}{2}, \frac{n}{2}, \frac{n}{2}}$ correspond to arrows crossing the central line.  For vertices of the form $(x, \frac{n+1}{2})$, when $d$ is even associate $f_{\frac{n+1}{2}, \frac{n+1}{2} -1, \frac{n+1}{2}-1}$ to the vertical arrow and $g_{\frac{n+1}{2}, \frac{n+1}{2} -1, \frac{n+1}{2}}$ to the diagonal arrow.  When $d$ is odd, it is necessary to make this assignment in the quiver $Q_{P, n}$, since $\tau$ swaps vertical and diagonal arrows.  The resulting assignment is compatible with $\tau$ by Equation (\ref{dmf6}).  This establishes the isomorphism of directed graphs between the Auslander-Reiten quiver and the candidate quivers.

With the above assignment of morphisms to the arrows of the Auslander-Reiten quiver, it follows directly from the composition formulas in Theorem \ref{dgstable morphism formula} that all squares commute, up to powers of $\psi$.

To determine when a composition is zero, it suffices to consider compositions starting at the vertex $(0, y)$, $1 \le y \le n$.  Suppose $d$ is even, and choose a composition of morphisms starting at $(0,y)$ and consisting of $s$ vertical arrows and $t$ diagonal arrows.  We may choose to represent the vertex $(0,y)$ by the module $M^{\overline{1}}_{\overline{y}}$; in this case a vertical arrow has as its unique nonzero component the canonical morphism $M^{\overline{1}}_{\overline{y}} \rightarrow M^{\overline{1}}_{\overline{y-1}}$ in $A\grstab$, and a diagonal arrow has unique nonzero component $M^{\overline{1}}_{\overline{y}} \rightarrow M^{\overline{n}}_{\overline{y}}(-d)$.  Thus a morphism consisting of $0 \le s < y$ vertical arrows and $0 \le t < n+1 -y$ diagonal arrows has the canonical morphism $M^{\overline{1}}_{\overline{y}} \rightarrow M^{\overline{1-t}}_{\overline{y-s}}(-d\delta_{t>0})$ as its unique  nonzero component.  When $s = y$, the composition becomes zero in $A\grmod$, and when $t= n+1-y$, it becomes zero in $A\grstab$; in either case, the morphism vanishes in $A\dgstab$.  For $s> y$ or $t > n+1-y$, the morphism factors through a morphism with $s=y$ or $t = n+1-y$, hence is zero.

When $d$ is odd, the automorphism $\tau$ sends vertical arrows to diagonal ones and vice-versa, hence the notion of "vertical" and "diagonal" arrows are only locally defined.  However, once a local choice of definition is made at the vertex $(0,y)$, the argument in the preceding paragraph applies without change.
\end{proof}

\section{Acknowledgement}
I would like to offer my heartfelt thanks to Rapha\"el Rouquier for suggesting the topic of this paper, for providing constant support and encouragement during its writing, and for teaching me homological algebra.

I would also like to thank Torkil Stai for helpful comments regarding the universal property of the triangulated hull, which led to a substantial revision (and hopefully improvement!) of Section \ref{The Dg-Stable Category}.

\bibliography{References}{}
\bibliographystyle{plain}

\end{document}